\documentclass[12pt,reqno]{amsart}
\usepackage{etoolbox}

\DeclareMathAlphabet\EuScript{U}{eus}{m}{n}
\DeclareMathAlphabet\EuScriptBold{U}{eus}{b}{n}
\DeclareMathAlphabet\Eurm{U}{eur}{m}{n}
\DeclareMathAlphabet\Eurb{U}{eur}{b}{n}
\newcommand{\mathcalb}{\EuScriptBold}

\usepackage{mathtools}
\usepackage{etoolbox}
\usepackage{url}

\DeclareFontFamily{U}{matha}{\hyphenchar\font45}
\DeclareFontShape{U}{matha}{m}{n}{
	<-6> matha5 <6-7> matha6 <7-8> matha7
	<8-9> matha8 <9-10> matha9
	<10-12> matha10 <12-> matha12
}{}
\DeclareSymbolFont{matha}{U}{matha}{m}{n}

\DeclareFontFamily{U}{mathx}{\hyphenchar\font45}
\DeclareFontShape{U}{mathx}{m}{n}{
	<-6> mathx5 <6-7> mathx6 <7-8> mathx7
	<8-9> mathx8 <9-10> mathx9
	<10-12> mathx10 <12-> mathx12
}{}
\DeclareSymbolFont{mathx}{U}{mathx}{m}{n}

\DeclareMathDelimiter{\vvvert} {0}{matha}{"7E}{mathx}{"17}%

\DeclarePairedDelimiterX{\normiii}[1]
{\vvvert}
{\vvvert}
{\ifblank{#1}{\:\cdot\:}{#1}}
\newcommand\scal[2]{\left\langle #1,#2\right\rangle}

\let\CMcal=\mathcal

\def\HH{\mathcal{H}}
\def\BH{\mathcalb{B}(\mathcal{H})}
\def\BXY{\mathcalb{B}(X,Y)}
\def\BXYz{\mathcalb{B}(X,Y^*)}

\def\M{\mathfrak{M}}

\def\N{\mathbb{N}}
\def\R{\mathbb{R}}
\def\C{\mathbb{C}}

\newcommand{\cci}{ {\mathcalb  C}_\iii({\mathcal H}) }

\newcommand{\GG}{\mathbb{G}}

\newcommand{\tr}{\operatorname{tr}}

\newcommand{\sgn}{\operatorname{sgn}}

\newcommand{\AAA}{ {\mathscr A} }
\newcommand{\AAt}{ {\mathscr A}_t }
\newcommand{\BBB}{ {\mathscr B} }
\newcommand{\CCC}{ {\mathscr C} }

\newcommand{\XzzYz}{X^{**}\widehat{\otimes}_\pi Y^*}

\newcommand{\WWW}{\Omega}

\newcommand{\iii}{\infty}

\newcommand{\sumn}{\sum_{n=1}^\iii}

\newcommand{\KXY}{\mathcalb{K}(X,Y)}

\newcommand{\ccc}{ {\mathcalb C}}

\newcommand{\ccj}{ {\mathcalb  C}_1({\CMcal H}) }

\newcommand{\PP}{\mathbb{P}}

\newcommand{\Btn}{\BBB_t^{(n)}}
\newcommand{\Stn}{\mathcal{S}_t^{(n)}}

\usepackage{mathtools} 
\usepackage{fixdif} 
\usepackage{leftindex}

\makeatletter
\patchcmd\maketitle
{\uppercasenonmath\shorttitle}
{}
{}{}
\patchcmd\maketitle
{\@nx\MakeUppercase{\the\toks@}}
{\the\toks@}
{}
{}{}
\patchcmd\@settitle{\uppercasenonmath\@title}{\Large}{}{}
\patchcmd\@setauthors
{\MakeUppercase{\authors}}
{\authors}
{}{}
\makeatother

\usepackage{amsmath,amssymb,amsthm, amsfonts}
\usepackage{color}
\usepackage{url}
\usepackage{tikz-cd}
\usepackage{combelow}
\input{mathrsfs.sty}
\usepackage[utf8]{inputenc}
\usepackage[T1]{fontenc}
\usepackage{enumerate}
\newtheorem{theorem}{Theorem}[section]

\newtheorem{definition}{Definition}[section]

\newtheorem{corollary}{Corollary}[section]
\newtheorem{proposition}{Proposition}[section]

\newtheorem{lemma}{Lemma}[section]
\newtheorem{remark}{Remark}[section]
\newtheorem{example}{Example}[section]
\numberwithin{equation}{section}

\usepackage[colorlinks=true]{hyperref}
\hypersetup{urlcolor=blue, citecolor=red , linkcolor= blue}

\begin{document}

      \title[Strongly Integrable Operator-Valued Functions...]{Strongly Integrable Operator-Valued Functions, Generated Vector Measures and \\ Compactness of Integrals}
		

\author[Milo\v s Arsenovi\' c]{Milo\v s Arsenovi\' c}
		\address{Faculty of Mathematics, University of Belgrade, Studentski trg 16, Belgrade, Serbia
		}
	
	\email{\url{milos.arsenovic.bg.ac.rs}}
	
		\author[Mihailo Krsti\' c]{Mihailo Krsti\' c}
		\address{Faculty of Mathematics, University of Belgrade, Studentski trg 16, Belgrade, Serbia
		}
		\email{\url{mihailo.krstic@matf.bg.ac.rs}}
        
		\author[Matija Milovi\' c]{Matija Milovi\' c}
		\address{Faculty of Mathematics, University of Belgrade, Studentski trg 16, Belgrade, Serbia
		}
	
	\email{\url{matija.milovic@matf.bg.ac.rs}}

    \author[Stefan Milo\v sevi\' c]{Stefan Milo\v sevi\' c}
		\address{Faculty of Mathematics, University of Belgrade, Studentski trg 16, Belgrade, Serbia
		}
	
	\email{\url{stefan.milosevic.bg.ac.rs}}


		
		
		\begin{abstract} Gel'fand integral of a family of compact operators on a Hilbert space is not always compact, even with additional property of positivity and commutativity. 
        We prove that integrals of a family, consisting of compact operators, in the space $L_{s}^1(\Omega,\mu,\mathcalb B(X, Y))$ of strongly integrable families are compact whenever $X$ does not contain an isomorphic copy of $\ell^1$.
       In addition, we prove an integral inequality for spectral radius
$$r\left(\int_\WWW\AAA \,d\mu\right)\leqslant\int_\WWW r(\AAA_t)\,d\mu(t)$$
for a mutually commuting family $\AAA$ in $L_s^1(\Omega,\mu,\mathcalb B(X))$, which generalizes a recent result obtained under a stronger assumption of Bochner integrability.
We prove also approximation results in $L_s^1(\Omega,\mu,\mathcalb B(X))$ in the case $X$ has finite dimensional Schauder decomposition.
All these results are based on a key theorem of this paper which states that every function in $L_{s}^1(\Omega,\mu, \mathcalb B(X, Y))$ generates a countably additive, in operator norm, $\mathcalb B(X, Y)$-valued measure whenever $X^*$ does not contain an isomorphic copy of $\mathfrak{c}_0$.
		\end{abstract}
		
		\maketitle			

\section{Introduction}

\subsection{Motivation, organization of the paper and notation}

Research on integration in Banach spaces began in early 20th century, based on theory of Lebesgue integral. A systematic and modern development started in the 1930s, when Salomon Bochner introduced the Bochner integral, providing a framework for integrating Banach space–valued functions, see \cite{Bochner1933}. 
Shortly thereafter, in the late 1930s, B. J. Pettis developed the Pettis integral, which is based on weak measurability, and extended the theory to broader classes of functions, see \cite{Pettis1938}. 
During the 1940s and 1950s, further contributions by mathematicians such as Dunford, Schwartz, and Gelfand firmly established integration in Banach spaces as a distinct topic in functional analysis, which has important applications to operator theory and probability theory. A standard reference for the general theory of integration in Banach spaces is a classical book \cite{DU}. For a short and clear introduction to this theory we refer reader to \cite[Sections 11.8,\,11.9,\,11.10]{Alip}.

Over the past two decades considerable attention has been given to integration of operator-valued (o.\,v.) functions, where integrability is defined by integrability of scalarizations (see Subsections ~\ref{OV}, \ref{OVH}). 
This line of research, where operators act on a Hilbert space $\HH$, can be found in a pioneering paper \cite{J05} and later in \cite{ JKL20, MMS, Matija, MK}. This theory has applications, for instance, in the study of spectral and numerical radii of integrals of operator-valued functions; see \cite{HMspectralr, HMnumrad}.
In this paper, we study functions taking values in the space $\BXY$, where $X$ and $Y$ are Banach spaces. Spaces of such o.\,v. functions were investigated in \cite{MAMK, BJN, OBIGB}. 
Recently, the study of integrable functions has been extended from operator-valued settings to functions with values in spaces of measures, see \cite{ABK1,ABK2}.

In this paper we continue investigation, initiated in \cite{BJN}, of a normed space $L^1_s(\Omega, \M, \mu, \mathcalb B(X, Y))$, which is wider than the space $L^1(\Omega, \M, \mu, \mathcalb B(X, Y))$ of Bochner integrable $\mathcalb B(X, Y)$-valued functions. 
It fits into a scale of normed spaces $L^p_s(\Omega, \M, \mu, \mathcalb B(X, Y))$, $1 \leqslant p < +\infty$, see \cite{BJN}. 
Finiteness of the corresponding norm was built into the definition of these spaces, but this finiteness follows from the assumption of strong $p$-integrability, see Proposition \ref{pBnormProp} below.
It was proved in \cite{BJN}, assuming $\mu(\Omega) < +\infty$, that a family in $L^p_s(\Omega, \M, \mu, \mathcalb B(X, Y))$ generates, via integration over measurable subsets of $\Omega$, a countably additive operator valued measure whenever $1 < p < +\infty$ and also for $p=1$ if a uniform integrability condition holds.
Here we prove that both conditions of uniform integrability and finiteness of $\mu$ can be dropped if $X^*$ does not contain an isomorphic copy of $\mathfrak{c}_0$, see Theorem \ref{ACmuPb} below. 
In fact, uniform integrability is always satisfied for families in $L_s^1(\Omega, \mu, \mathcalb B(X, Y))$ provided $X$ contains no copy of $\ell^1$, see Corollary \ref{UINT}.
Theorem \ref{ACmuPb} is a key result of the paper, it is used to prove that a strongly integrable family of compact operators on a separable Hilbert space is Pettis integrable, and hence its integrals over measurable subsets are compact, see Corollary \ref{petisCii} and Theorem \ref{Cpt} for a stronger statement.
In the last section we apply developed theory to the spectral radius inequality, see Theorem \ref{thradijus}. 

Throughout this paper $(\Omega,\M,\mu)$ stands for a measurable space with a complete measure $\mu$. 
If $\Omega$ is an interval in $\R$, then we take $\mu$ to be the Lebesgue measure $m$ on Lebesgue measurable subsets of $\Omega$. 
If $\Omega = \N$, then $\mu$ is assumed to be the counting measure on $\mathcal{P}(\N)$. We often omit $\M$ and $\mu$ from notation, especially in the above special cases of an interval and $\N$.  Disjoint unions are denoted by $\bigsqcup$.

In this paper $X$, $Y$ and $Z$ denote Banach spaces, $\|\cdot\|_X$ denotes norm in $X$, $B_X$ denotes the closed unit ball of $X$ and topological dual of $X$ is denoted by $X^\ast$. We shall use the notation $\scal{x}{x^*}$ to denote the duality pairing of the elements $x\in X$ and $x^*\in X^*$.
The space of bounded linear operators from $X$ to $Y$ will be denoted by $\mathcalb B(X, Y)$ and for $A\in\BXY$ the usual operator norm is denoted by $\|A\|$. Also, we set $\mathcalb B(X) = \mathcalb B(X, X)$.  The adjoint of $A \in \mathcalb B(X, Y)$ is denoted by $A^*\in\mathcalb B(Y^*, X^*)$.
The space of \textit{compact} linear operators from $X$ to $Y$ is denoted by $\KXY$.

For $a$ in $Y$ and $b^*$ in $X^*$, rank-$1$ operator $a\otimes b^*$ in $\mathcalb B(X, Y)$ is defined by: $(a \otimes b^*)(x) = \scal{x}{b^*} a$, $x \in X$. 
An equality $\| a \otimes b^* \| = \| a \|_Y \cdot \| b^* \|_{X^*}$ is easily verified. 
Also, the next formula will be used later in the paper:
\begin{equation}\label{tesor} 
(a \otimes b^*) \circ A = a \otimes A^*b^*, \qquad A \in \mathcalb B(X, Y), \quad a \in X, \quad b^* \in Y^*.
\end{equation}
Its verification is straightforward, for all $x\in X$ we have: 
$$((a \otimes b^*) \circ A)x = (a \otimes b^*) (Ax) = \langle Ax, b^* \rangle a = \langle x, A^*b^* \rangle a = (a \otimes A^*b^*) x.$$

The spectrum of an operator $A \in \mathcalb B(X)$ is denoted by $\sigma(A)$ and its spectral radius by $r(A)$. 
An operator $A \in \mathcalb B(X)$ is said to be Volterra if it is compact and $\sigma (A) = \{ 0 \}$.

$\HH$ denotes a separable Hilbert space. 
The space of compact linear operators on $\HH$ is denoted by $\cci$ and the space of nuclear operators on $\HH$ is denoted by $\mathcalb C_1(\HH)$. 

\subsection{Preliminaries from Banach space theory}\label{Pgr}

A series $\sum_{n=1}^\infty x_n$ in $X$ is said to be \textit{unconditionally convergent} if it converges independently of the order of its terms, that is, for every permutation $\pi:\mathbb{N}\to\mathbb{N}$ the series $\sum_{n=1}^\infty x_{\pi(n)}$ converges in $X$.
A series $\sum_{n=1}^{\infty}x_n$ in $X$ is said to be \textit{weakly unconditionally Cauchy} if $\sum_{n=1}^\infty |\scal{x_n}{x^*}| < +\infty$ for all $x^*$ in $X^*$. The space of all such sequences is denoted by $\ell^1_w(X)$ with a norm given by
$$\left\|(x_n)_{n=1}^\iii \right\|_{\ell^1_w(X)} = \sup_{\| x^\ast \|_{X^*} = 1} \sum_{n=1}^\infty | \langle x_n, x^\ast \rangle |.$$
In the case where $X$ is a dual space, the following lemma holds.
\begin{lemma}\cite[Section~8.2\,]{DF}\label{ell1slslz}
We have the next equality:
$$\ell_w^1(X^*) = \left\{(x_n^*)_{n=1}^\infty:\N\to X^* \,\,\Bigg|\,\,\sum_{n=1}^\infty |x_n^*(x)| < +\infty \text{ for all } x \in X \right\}.$$
\end{lemma}

In the case of $X=Y=\HH$, the space $\ell_{w}^{1} (\HH)$ consists of all sequences $(f_n)_{n=1}^\iii$ in $\HH$ such that $(\scal{f_n}{f})_{n=1}^\iii$ belongs to the space $\ell^1$ for all $f \in \HH$. 

It is said that a Banach space $X$ has cotype $q$, for $2\leqslant q <+\infty$, if there is a constant $\kappa > 0$ such that for every $n\in\N$ and $x_1,x_2,\ldots,x_n\in X$ holds
\begin{equation*}
\left(\sum_{k=1}^n\|x_k\|_X^q\right)^\frac{1}{q}\leqslant\kappa\cdot\left(\int_0^1\left\|\sum_{k=1}^nr_k(t)x_k\right\|^2_Xdt\right)^\frac{1}{2},
\end{equation*}
where $r_n(t) = \sgn(\sin(2^{n+1}\pi t))$ are Rademacher functions.
The smallest of all such constants $\kappa$ will be denoted by $C_q(X)$. By \cite[Proposition 14.6]{DJT} we get the following theorem.
\begin{theorem}\label{kotipniz}
    Let $X$ be a Banach space that has cotype $q$. Then, $\ell^1_w(X)\subseteq\ell^q(X)$ and for every $(x_n)_{n=1}^\iii\in\ell^1_w(X)$ holds 
    \begin{equation*}
     \|(x_n)_{n=1}^\iii\|_{\ell^q(X)}\leqslant C_q(X)\cdot\|(x_n)_{n=1}^\infty\|_{\ell^1_w(X)}   
    \end{equation*}
\end{theorem}

\begin{remark}\label{kotiptac}
    As noted in the proof of \cite[Proposition 14.6]{DJT}, Theorem \ref{kotipniz} states that $I_X$ is $(q,1)$-summing and from the definition of $(q,1)$-summing norm we have $\pi_{q,1}(I_X)\leqslant C_q(X)$. Moreover, the best constant in Theorem \ref{kotipniz} is $\pi_{q,1}(I_X)$.
\end{remark}

Specially, since the Hilbert space $\HH$ has a cotype 2 constant $C_2(\HH) = 1$, the following embedding theorem holds (details can be found in \cite[Section~8.\,9]{DF}).

\begin{theorem}\label{Pi21}
For every Hilbert space $\HH$ we have $\ell^1_w(\HH)\subseteq\ell^2(\HH)$, and for every $(x_n)_{n=1}^\iii\in\ell^1_w(\HH)$ we have $\| (x_n)_{n=1}^\iii \|_{\ell^2(\HH)} \leqslant \| (x_n)_{n=1}^\iii \|_{\ell^1_w(\HH)}$.
\end{theorem}

We record here two results from Banach space theory that we need. 
Let us recall that a subset $C$ of a Banach space $X$ is \textit{conditionally weakly compact} if every sequence in $C$ has a weakly Cauchy subsequence. The following Rosenthal $\ell^1$ Theorem can be found in \cite[Page 201]{D}.
\begin{theorem}\label{CWC}
 A Banach space $X$ does not contain a copy of $\ell^1$ if and only if every bounded subset of $X$ is conditionally weakly compact.
\end{theorem}
The next proposition is an easy consequence of Theorem \ref{CWC}, for the reader's convenience we include a proof from \cite{RosentalArxiv}. An extensive list of properties equivalent to "$X$ \textit{does not contain an isomorphic copy of} $\ell^1$" can be found in \cite{RosentalArxiv}.
\begin{proposition}\label{CCjeK}
    If $X$ does not contain a copy of $\ell^1$, every completely continuous operator $T\in\mathcalb{B}(X,Y)$ is compact, for every Banach space $Y$.
\end{proposition}
\begin{proof}
Let $T\colon X\rightarrow Y$ be completely continuous  and let $(x_n)_{n=1}^\infty$ be a bounded sequence in $X$. Due to Theorem \ref{CWC} there exists a weakly Cauchy subsequence $(x_{n_k})_{k=1}^\infty$. For any strictly increasing sequences $(k_i)_{i=1}^\infty$ and $(m_i)_{i=1}^\infty$ of positive integers, the sequence $(x_{n_{k_i}}-x_{n_{m_i}})_{i=1}^\infty$ is weakly null, and by the hypothesis it follows that $\|T(x_{n_{k_i}}-x_{k_{m_i}})\|_Y = \|Tx_{n_{k_i}}-Tx_{k_{m_i}}\|_Y \to 0$ as $i\to\infty$. Hence $(Tx_{n_k})_{k=1}^\infty$ is a Cauchy sequence in $Y$ and thus convergent, since $Y$ is complete. Therefore, $T\in\KXY$, completing the proof. 
\end{proof}
The other result we need is the following theorem.
\begin{theorem}[Odell-Rosenthal Theorem]\label{OdellRosenthal}
   Let $X$ be a separable Banach space that does not contain $\ell^1$. Then, for every $x^{**}\in B_{X^{**}}$ there exists a sequence $(x_n)_{n=1}^\iii$ of elements in $B_X$ such that $\lim\limits_{n\rightarrow\iii}\scal{x_n}{x^*}=\scal{x^*}{x^{**}}$ for every $x^*\in X^*$.
\end{theorem}

Many of our results depend on the following properties of Banach space $X$. The first one is "$X$ does not contain an isomorphic copy of $\ell^1$", the other one is "$X^*$ does not contain an isomorphic copy of $\mathfrak{c}_0$". There is a relation between these properties: the second one follow from the first one.
This follows from Theorem \ref{c0l1} below due to Bessaga and Pelczinsky, see \cite[Theorem 10, Page 48]{D}.
\begin{theorem}\label{c0l1}
    A Banach space $X$ contains a complemented copy of $\ell^1$ if and only if $X^*$ contains a copy of $\mathfrak{c}_0$.
\end{theorem}

Let us review basic definitions and notation from the theory of vector measures, for more details we refer to a standard reference \cite{DU}. 
       A mapping $\nu\colon\M\rightarrow X$ is a \textit{vector measure} if $\nu(\emptyset)=0$ and $\nu(A\sqcup B)=\nu(A)+\nu(B)$ for disjoint sets $A,B\in\M.$
By induction, every vector measure is finitely additive. If, in addition, for every sequence $(E_n)_{n=1}^\iii$ of disjoint sets in $\M$ we have $\nu(\bigsqcup_{n=1}^\iii E_n)=\sum_{n=1}^{\infty}\nu(E_n)$ in the norm of $X$, we say that $\nu$ is \textit{countably additive} (or $\sigma$-additive).
Moreover, we say that a family of vector measures $\mu_\lambda\colon\M\to X$ for $\lambda\in\Lambda$ is \textit{uniformly  countably additive} if
\begin{equation}\label{UCAdef}
\lim_{n\rightarrow\iii}\sup_{\lambda\in\Lambda}\left\|\mu_\lambda\left(\bigsqcup_{k=n+1}^\iii E_k\right)\right\|_X = 0,
\end{equation}
for all sequences of disjoint subsets $(E_n)_{n=1}^\iii$. Equivalent condition is:
\begin{equation}\label{UCAekv}
\lim_{n\rightarrow\iii}\sup_{\lambda\in\Lambda}\|\mu_\lambda(F_n)\|_X=0
\end{equation}
for all decreasing sequences $(F_n)_{n=1}^\iii$ of measurable subsets such that $F_n\downarrow\emptyset$.

Since we are interested in integration with respect to a (positive countably additive) measure $\mu$, we recall the following standard definition.
    Let $(\WWW,\M,\mu)$ be a measure space and $\nu\colon\M\rightarrow X$ be a vector measure. It is said that $\nu$ is \textit{$\mu$-continuous}, denoted by $\nu\ll\mu,$ if for every $\varepsilon>0$ there exists $\delta>0$ such that $E\in\M$ and $\mu(E)<\delta$ imply $\|\nu(E)\|_X < \varepsilon$.
Every $\mu$-continuous vector measure obviously vanishes on sets of $\mu$-measure zero, while the converse does not hold in general. 
We note that the assumption $\mu(\Omega) < +\infty$ is not needed in the proof of the following theorem, which can be found in \cite[Theorem~1, Page 10]{DU}.
\begin{theorem}\label{petismu}
Let $(\WWW,\M,\mu)$ be a measure space and $\nu\colon\M\rightarrow X$ be a countably additive vector measure. Then, $\nu\ll\mu$ if and only if $\nu$ vanishes on sets of $\mu$-measure zero.
\end{theorem}

Let us review notions of measurability of functions taking values in $X$ or $X^\ast$. 
Recall that $\varphi : \Omega \to X$ is said to be \textit{strongly $\mu$-measurable} if there is a sequence $s_n : \Omega \to X$ of $\mu$-simple functions such that $\lim_{n \to \infty} s_n(t) = \varphi (t)$ for $\mu$-a.\,e. $t\in\Omega$. 
Here $\mu$-simple function means a function $s$ of the form $s = \sum_{k=1}^m x_k \chi_{A_k}$, where $(A_k)_{k=1}^m$ is a sequence of disjoint $\mu$-measurable subsets of $\Omega$ of finite $\mu$ measure and $(x_k)_{k=1}^m$ is a sequence in $X$. The vector space of all strongly $\mu$-measurable functions is denoted by 
$\mathcal{M}(\WWW,\M,\mu,X)$, with the usual convention of identification of a.\,e. equal functions. It is clear that for every strongly $\mu$-measurable $\varphi : \Omega \to X$ the set $\{ t \in \Omega : \varphi(t) \not= 0 \}$, called support of $\varphi$, has $\sigma$-finite measure. If $X=\C$, we refer to such functions as $\mu$-measurable. 
For every strongly $\mu$-measurable function $\varphi$, the mapping $\Omega\ni t\mapsto\|\varphi(t)\|_X\in\R$ is $\mu$-measurable.
A function $\varphi : \Omega \to X$ (resp. $\varphi : \Omega \to X^\ast$) is said to be \textit{weakly}
(resp. \textit{weakly$^*$}) \textit{$\mu$-measurable} if for every $x^\ast \in X^\ast$ (resp. $x \in X$) the function $\Omega\ni t \mapsto \langle
\varphi (t), x^\ast \rangle\in\C$ (resp. $\Omega\ni t \mapsto \langle x, \varphi (t) \rangle\in\C$) is $\mu$-measurable.
We denote these functions by $\langle \varphi, x^\ast \rangle$ (resp. $\langle x, \varphi \rangle$).

In principle, there are two ways to identify functions $\varphi, \psi : \Omega \to X^*$ (resp. $X$). The first one is the usual a.\,e. concept: there is a null set $S \subset \Omega$ such that $\varphi(t) = \psi(t)$ for all $t \in \Omega\setminus S$. 
The second one is that for every $x \in X$ (resp. $x^* \in X^*$) there is a null set $S_x \subset \Omega$ (resp. $S_{x*} \subset \Omega$) such that $\langle x, \varphi(t) \rangle = \langle x, \psi(t) \rangle$ for all $t \in \Omega \setminus S_x$ (resp. $\langle \varphi(t), x^* \rangle = \langle \psi(t), x^* \rangle$ for all $t \in \Omega \setminus S_{x*}$).
If the space $X$ is separable, these two types of identifications are equivalent.

The following theorem is a basic result concerning connection of strong and weak $\mu$-measurability. For more details see \cite[Theorem~1.1.20.]{ABS}.
\begin{theorem}[Pettis Measurability Theorem]\label{PMT}
    A function $\varphi : \Omega \to X$ is strongly $\mu$-measurable if and only if  $\varphi$ is $\mu$-weakly measurable and there exist $E\in\M$ such that $\mu(E)=0$ and $\varphi (\Omega\setminus E) \subset X$ is separable.
\end{theorem}
If $X$ is separable, then for every weakly $\mu$-measurable $\varphi : \Omega \to X$ the function $\Omega\ni t \mapsto \| \varphi(t) \|_X \in\R$ is $\mu$-measurable. Indeed, by separability of $X$ and Theorem \ref{PMT}, $\varphi$ is strongly $\mu$-measurable and thus  $\Omega\ni t \mapsto \| \varphi(t) \|_X\in\R$ is $\mu$-measurable.  
Similarly, if $\varphi : \Omega \to X^\ast$ is weakly$^*$ $\mu$-measurable, then $\Omega\ni t \mapsto \| \varphi(t) \|_{X^*} \in\R$ is $\mu$-measurable function. Indeed, $\| \varphi(t) \|_{X^*} = \sup\{| \langle x, \varphi(t) \rangle |:x\in S\},$ where $S$ is a countable dense subset of the unit sphere in $X$. 

Next we review notions of integrability of functions taking values in $X$ or $X^\ast$. 
If for a weakly $\mu$-measurable function $\varphi : \Omega \to X$ we have $\langle \varphi, x^\ast \rangle\in L^1(\Omega, \mu)$ for every $x^\ast$ in $X^\ast$, then it is said that $\varphi$ is \textit{Dunford integrable} and we denote the vector space of all Dunford integrable $X$-valued functions by $\mathbb D(\Omega, \M, \mu, X)$. 
In that case for every $E\in\M$ there is a unique element $\varphi_E$ in $X^{\ast \ast}$ such that 
$\langle x^\ast, \varphi_E \rangle = \int_E \langle \varphi, x^\ast \rangle \,d \mu$ for every $x^\ast$ in $X^\ast$, and we write 
$\varphi_E = \prescript{\mathrlap{D}\,\,}{}{\int_{E}} \varphi \, d\mu$.
It may happen (and it always happens if $X$ is reflexive) that $\varphi_E$ is in $X$ for every measurable $E$. In that case it is said that $\varphi$ is \textit{Pettis integrable} and we write 
$\varphi_E = \prescript{\mathrlap{P}\,\,}{}{\int_{E}} \varphi \, d\mu$.
The space of all Pettis integrable functions $\varphi : \Omega \to X$ is denoted by $\PP(\Omega, \M, \mu, X)$, where $\varphi$ is identified with $\psi$ whenever $\langle \varphi, x^* \rangle = \langle \psi, x^* \rangle$ $\mu$-a.\,e. for all $x^* \in X^*$. A norm is defined on this space by formula:
\begin{equation}\label{Pnorm}
    \| \varphi \|_{\PP} = \sup_{\| x^* \|_{X^*} \leqslant 1} \int_\Omega | \langle \varphi(t), x^* \rangle| \, d\mu(t)
\end{equation}
Analogously, if for a weak$^*$ $\mu$-measurable function $\varphi : \Omega \to X^\ast$ the function $\langle x, \varphi \rangle$ is in $L^1(\Omega, \mu)$ for every $x \in X$, then it is said that $\varphi$ is \textit{Gelfand integrable}.
In that case for every measurable $E$ there is a unique $\varphi^E$ in $X^\ast$ such that for every $x \in X$ we have $\left\langle x, \varphi^E \right\rangle = \int_E \langle x, \varphi \rangle \,d\mu$ and we set 
$ \varphi^E = \prescript{\mathrlap{G}\,\,}{}{\int_{E}} \varphi\, d\mu$. 
The space of all such functions is denoted by $\GG(\WWW,\M,\mu,X^*)$, assuming the following identification: $\varphi \sim_\GG \psi$ if and only if $\langle x, \varphi \rangle = \langle x, \psi \rangle$ $\mu$-almost everywhere for every $x \in X$. 
Then a norm on $\GG(\WWW,\M,\mu,X^*)$ is defined by: 
$$\| \varphi \|_{\GG(\Omega, \M, \mu, X^*)} = \displaystyle\sup_{\| x \|_X = 1} \| \langle x, \varphi \rangle \|_{L^1(\mu)}.$$
If $X$ is separable, the usual a.\,e. identification and $\sim_\GG$ identification coincide. 
For $\varphi\in\GG(\WWW,\M,\mu,X^*)$ we define $\nu_\varphi\colon\M\rightarrow X^*$ by
\begin{equation}\label{meranu}
\nu_\varphi(E) = \int_E \varphi\,d\mu,\qquad E \in \M.
\end{equation}
Clearly $\nu_\varphi$ is a vector measure. Moreover, $\nu_\varphi$ is weakly$^*$ countably additive, i.e.  if $E=\bigsqcup_{n=1}^\iii E_n$ for a disjoint sequence $(E_n)_{n=1}^\iii$ of sets in $\M$, then for all $x\in X$
$$\langle x,\nu_\varphi(E)\rangle=\int_{E}\scal{x}{\varphi(t)}d\mu(t)=\sum_{n=1}^\iii\int_{E_n}\scal{x}{\varphi(t)}d\mu(t)=\sumn\scal{x}{\nu_\varphi(E_n)}.$$
The second equality above follows from integrability of scalar function $\langle x, \varphi \rangle$.  Also, if $\mu(E)=0$, then $\scal{x}{\nu_\varphi(E)} = \int_E\scal{x}{\varphi(t)}d\mu(t) = 0$ for $x\in X$ and hence $\nu_\varphi(E)=0$. Thus, we have the following proposition.
\begin{proposition}\label{wkstcameasure}
For every $\varphi\in\GG(\WWW,\M,\mu,X^*)$ the mapping 
$\nu_\varphi\colon\M\rightarrow X^*$ defined by \eqref{meranu} is a weakly$^*$ countably additive vector measure that vanishes on sets of $\mu$-measure $0$.
\end{proposition}

A strongly $\mu$-measurable function $\varphi : \WWW \to X$ is said to be \textit{Bochner $p$-integrable}, for $1 \leqslant p < +\infty$, if $\int_\Omega \| \varphi (t) \|_X^p\, d\mu(t) < +\infty$. These functions form a Banach space $L^p (\Omega, \M, \mu, X)$ with a natural norm, assuming the usual a.\,e. identification. 
Also, the vector space of $\mu$-simple functions is dense in $L^p (\Omega, \M, \mu, X)$.
If $\Omega=\N$ 
then $\ell^p(X)=\{f\colon\N\rightarrow X:\sum_{n=1}^\iii\|f(n)\|_X^p<+\iii\}$ is the space of Bochner $p$-integrable functions, as every $f\colon\N\rightarrow X$ is strongly $\mu$-measurable. 
Functions in $L^1 (\Omega, \M, \mu, X)$ are called \textit{Bochner integrable}.
For every $E \in \M$  and $\varphi$ in $L^1 (\Omega, \M, \mu, X)$ there is a naturally defined element $\prescript{\mathrlap{B}\,\,}{}{\int_{E}} \varphi \, d\mu$ in $X$, the Bochner integral of $\varphi$ over $E$.
Then $\nu_\varphi : \M \to X$ defined by $\nu_\varphi (E) = \prescript{\mathrlap{B}\,\,}{}{\int_{E}} \varphi\,d\mu$ is a countably additive vector measure on $(\Omega, \M)$ which is $\mu$-continuous. 

We alert the reader that there are weaker forms of $p$-integrability in the spirit of Dunford and Pettis, see \cite{jma2018}.
It is easily seen that Bochner integrability implies Pettis integrability, Pettis integrability directly implies Dunford integrability. 
It is also easily seen that Bochner integrability of $\varphi : \Omega \to X^\ast$ imples Gelfand integrability of $\varphi$.
We will often, for simplicity, suppress $\M$ from notation and write, for example, $L^1(\Omega, \mu, X)$ instead of $L^1(\Omega, \M, \mu, X)$.

\subsection{$\mathcalb B(X, Y)$-valued functions}\label{OV}

In this paper we deal with $\BXY$-valued functions, some of the results, examples and counterexamples will be specific to the Hilbert space case, i.\,e. to $\BH$-valued functions. 
For an o.\,v. function (equivalently family) $\AAA = (\AAA_t)_{t\in\Omega}$ in $\mathcalb B(X, Y)$ the family $(\AAA_t^*)_{t\in\Omega}$ in $\mathcalb B(Y^*, X^*)$ will be denoted by $\AAA^*$. 
For $x$ in $X$ the family $(\AAA_t x)_{t \in \Omega}$ in $Y$ is denoted by $\AAA x$.
The pointwise product of families $\AAA$ and $\BBB$ is denoted by $\AAA\BBB$, whenever this is well defined. 
We say a family $\AAA$ is bounded if $\sup_{t\in\Omega}\|\AAA_t\|<+\infty$.
 
A family $(\AAA_t)_{t\in\Omega}$ is said to be \textit{pointwise} (resp. \textit{pointwise weakly}) $\mu$-\textit{measurable} if $\AAA x$ is strongly (resp. weakly) $\mu$-measurable for all $x$ in $X$. 
Two families $\AAA$ and $\BBB$ are said to be \textit{pointwise equivalent} (resp. \textit{pointwise weakly equivalent}) if for every $x$ in $X$ we have $\AAA_t x = \BBB_t x$ a. e.
(resp. for every $(x, y^*) \in X \times Y^*$ we have $\langle \AAA_t x, y^* \rangle = \langle \BBB_t x, y^* \rangle$ a. e.). We denote the vector space of all classes of pointwise equivalent pointwise $\mu$-measurable (pointwise weakly equivalent pointwise weakly $\mu$-measurable) families by $\mathcal M_s(\Omega, \M, \mu, \BXY)$ (resp. $\mathcal M_w(\Omega, \M, \mu, \BXY)$, see \cite{BJN}.
A family $\AAA = (\AAA_t)_{t \in \Omega}$ in $\mathcalb B(X, Y^*)$ is said to be \textit{pointwise weakly}$^*$ $\mu$-\textit{measurable} if $\AAA x$ is weakly$^*$ $\mu$-measurable for every $x \in X$. Two such families $\AAA$ and $\BBB$ are said to be \textit{pointwise weakly}$^*$ \textit{equivalent} if, for every $(x, y) \in X \times Y$ we have $\langle y, \AAA x \rangle = \langle y, \BBB x \rangle$ a.\,e. We denote the space of all equivalence classes of pointwise weakly$^*$ $\mu$-measurable functions by $\mathcal M_{w^*}(\Omega, \M, \mu, \mathcalb B(X, Y^*))$. 
By Theorem \ref{PMT}, if the space $Y$ is separable, we have the equality 
\begin{equation}\label{Mws}
\mathcal M_w(\Omega, \M, \mu, \BXY) = \mathcal M_s(\Omega, \M, \mu, \BXY).
\end{equation}

If $\AAA$ belongs to $\mathcal M_s(\Omega, \M, \mu, \BXY)$ and $X$ is separable, then the function $\Omega\ni t \mapsto \| \AAA_t \|\in\R$ is $\mu$-measurable. 
This follows from $\| \AAA_t \| = \sup_{x\in S} \| \AAA_t x \|_Y$, where $S$ is a countable dense subset of $B_X$.

If $Y$ is separable and does not contain an isomorphic copy of $\ell^1$, then by Theorem \ref{OdellRosenthal} we have
\begin{equation*}
    \mathcal M_{w^*}(\Omega, \M, \mu, \mathcalb B(X, Y^*)) = \mathcal M_w(\Omega, \M, \mu, \mathcalb B(X, Y^*)).
\end{equation*}

We say that $\AAA$ in $\mathcal M_w(\Omega, \M, \mu, \BXY)$ is \textit{pointwise Dunford integrable} if $\AAA x \in \mathbb D(\Omega, \M, \mu, Y)$ for all $x \in X$. 
$ L^1_w(\Omega, \M, \mu, \BXY)$ denotes the vector space of all such classes of weakly pointwise equivalent functions, it was defined in \cite{BJN}, see also \cite{MAMK} where notation $L^1_{pD}$ instead of $L^1_w$ was used. 
For every $\AAA$ in $L^1_w(\Omega, \M, \mu, \BXY)$ and $E \in \M$ there is $\int_E \AAA \, d\mu$ in $\mathcalb B(X, Y^{**})$ defined by:
\begin{equation*}\label{pDint}
\left( \int_E \AAA \, d\mu\right) x = \prescript{\mathrlap{D}}{}{\int_{E}} \AAA x \, d\mu, \qquad x \in X,
  \end{equation*}
(see \cite[Proposition~2.9]{MAMK} for details). Finally, we define
\begin{equation*}
\begin{split}
 &L^1_{w^*}(\Omega, \M, \mu, \mathcalb B(X, Y^*)) = \\ & = \{ \AAA \in \mathcal M_{w^*}(\Omega, \M, \mu, \mathcalb B(X, Y^*)): \langle y,\AAA x \rangle \in L^1(\mu), \,\,\forall(x, y) \in X \times Y \}.
 \end{split}
\end{equation*}
Due to duality $\left(X\widehat{\otimes}_\pi Y\right)^*\cong\mathcalb{B}(X,Y^*)$, see \cite[Chapter 3.2]{DF}, it turns out that 
$$L^1_{w^*}(\Omega, \M, \mu, \mathcalb B(X, Y^*))=\GG(\Omega, \M, \mu, \mathcalb B(X, Y^*)).$$

\subsection{$\BH$-valued functions}\label{OVH}

Note that $\ccj^*\cong\BH$ and $\ccc_\iii (\HH)^* = \ccc_1(\HH)$, where duality is trace duality. Hence $\AAA\colon\WWW\rightarrow\BH$ is weakly$^*$ $\mu$-measurable (Gelfand integrable) if and only if function $\Omega\ni t\mapsto\tr(\AAA_t X)\in\C$ is $\mu$-measurable (integrable) for all $X$ in $\ccj$. 
Identification a.\,e. and identification using functions $\Omega\ni t\mapsto\tr(\AAA_tX)\in\C$ lead to the same equivalence relation on $\GG(\WWW,\M,\mu,\BH)$ because $\ccj$ is separable, see Subsection \ref{Pgr}.
In order to prove weak$^*$ $\mu$-measurability (Gelfand integrability) it suffices to show $\mu$-measurability (integrability) of the induced sesquilinear form, i.\,e. functions $\Omega\ni t\mapsto\langle\AAA_tf,g\rangle\in\C$, for all $f,g\in\HH$ (see \cite[Lemma~1.1]{MMS}). 
Thus, we have 
$$\GG(\WWW,\M,\mu,\BH) = L^1_w(\Omega, \M, \mu, \BH).$$
Also, by polarization identity, it suffices to check $\mu$-measurability (integrability) of  $\Omega\ni t \mapsto \langle \AAA_t f, f \rangle\in\C$ for all $f$ in $\HH$. The function 
$\AAA : \Omega \to \BH$ is weakly$^*$ $\mu$-measurable (Gelfand integrable) if and only if $\AAA^\ast$ is weakly$^*$ $\mu$-measurable (Gelfand integrable).
By Theorem \ref{PMT}, weak$^*$ $\mu$-measurability of $\AAA$ is  equivalent to the strong $\mu$-measurability of $\AAA f: \Omega \to \HH$ for every $f \in \HH$.
Furthermore, as already noted in Subsection \ref{OV}, for
every such o.\,v. function, the mapping $\Omega\ni t\mapsto\|\AAA_t\|\in\R$ is $\mu$-measurable (see also \cite[Page~322]{J05}). 

\section{Preliminaries on $L^p_s(\Omega, \M, \mu, \mathcalb B(X, Y))$ spaces}

The spaces appearing in the title of this section were introduced in \cite{BJN}, under the additional assumptions of finiteness of the measure $\mu$ and finiteness of expression in \eqref{pBnorm};  see the discussion in that paper after Proposition 3.1. Here we revisit these spaces and collect several elementary and auxiliary results concerning them, which will serve as a preparation for the main results proved in Section \ref{MAIN}. 

\begin{definition}\label{OsnovnaDef}
Let $1\leqslant p<+\iii$ and $\AAA \colon \WWW \rightarrow \mathcalb B(X, Y)$ be a pointwise $\mu$-measurable function. 
We say that $\AAA$ is strongly $p$-integrable if  for all $x \in X$ holds $\int_\WWW \|\AAA_t x \|_Y^p\,d\mu(t) < +\iii$. The space of all such functions is denoted by $L^p_s(\Omega, \M, \mu, \mathcalb B(X, Y))$, with the identification of pointwise equivalent functions.
\end{definition}

If $\WWW=\N$ and $\mu$ is the counting measure on $\M=\mathcal{P}(\N),$ the corresponding space of strongly $p$-integrable functions is denoted by $\ell^p_s(\mathcalb B(X, Y))$ and this space is a Banach space (see Proposition \ref{discr} below).  Also,
if $p=1$ we say that $\AAA$ is \textit{strongly integrable}. In that case for every $E \in \M$ there is an operator $\int_E \AAA \, d\mu \in \mathcalb B(X, Y)$ defined by the formula
\begin{equation}\label{Ulazakx}
\left( \int_E \AAA d\mu \right) x = \prescript{\mathrlap{B}}{}{\int_{E}} \AAA_t x \,d\mu(t)\,\,\,
\text{ for every}\,\,\, x \in X.
\end{equation}
Moreover, for $\AAA\in L^1_s(\WWW,\M,\mu,\BXYz)$ and every $x\in X$ and $y\in Y$ we have $|\scal{y}{\AAA_tx}|\leqslant\|\AAA_tx\|_{Y^*}\cdot\|y\|_Y$, hence $\AAA\in\GG(\WWW,\M,\mu,\BXYz)$. In this case
$$\scal{y}{\left( \int_E \AAA\, d\mu \right) x}=\int_E \scal{y}{\AAA_t x}\, d\mu(t) = 
\scal{y}{\prescript{\mathrlap{B}}{}{\int_{E}} \AAA_t x\,d\mu(t)}, \qquad E \in \M.$$



We record here, for future use, the following simple lemma. 



\begin{lemma}\label{sigfin}
    If $X$ is separable, then every $\AAA$ in $L_s^p(\Omega,\mu,\BXY)$ vanishes outside a set of $\sigma$-finite measure.
\end{lemma}
\begin{proof}
    We choose a countable dense subset $S$ of $B_X$, since $\AAA x \in L^1(\Omega, \mu, Y)$ for every $x \in S$, there is a set $E_x$ of $\sigma$-finite measure such that $\AAA x$ vanishes outside $E_x$. 
    It is easily seen that $\bigcup_{x \in S} E_x$ is a required set of $\sigma$-finite measure.
\end{proof}
Example \ref{longline} below shows that separability of the space $X$ in Lemma \ref{sigfin} is essential, even if we work with Hilbert spaces only.
\begin{example}\label{longline}
    Let $\Lambda$ be an uncountable set, let $I_\lambda = [0, 1] \times \{ \lambda \}$ for $\lambda \in \Lambda$ and set $\Omega = \bigsqcup_{\lambda \in \Lambda} I_\lambda$. We transfer Lebesgue $\sigma$-algebra $\mathcal L$ from $[0,1]$ to $\sigma$-algebra $\mathcal L_\lambda$ on $I_\lambda$ and transfer Lebesgue measure $m$ from $[0, 1]$ to measure $m_\lambda$ on $I_\lambda$. 
    Then we have a $\sigma$-algebra $\M = \{ E \subset \Omega : E \cap I_\lambda \in \mathcal L_\lambda, \lambda \in \Lambda \}$ on $\Omega$ and a measure $\mu$ on $\M$ defined by $\mu(E) = \sum_{\lambda \in \Lambda} m_\lambda (E \cap I_\lambda)$ for all $E\in\M$. 
    Let $1 \leqslant p < + \infty$. Then $X = \ell^p(\Omega)$ is a non separable Banach space, with a natural family of unit vectors $\{e_{t, \lambda}:0 \leqslant t \leqslant 1,\lambda \in \Lambda\}$. We define, for $(t, \lambda) \in \Omega$, $\AAA_{t, \lambda}\in\mathcalb B(X)$ by $\AAA_{t, \lambda} x = x(t, \lambda) e_{t, \lambda}$. 
    For a given $x \in X$ there are only countably many $(t, \lambda) \in \Omega$ such that $\AAA_{t, \lambda} x \not= 0$, and therefore $\AAA x$ is a function on $\Omega$ that vanishes outside a countable set and hence $\AAA x$ is $\mu$-Bochner integrable. This means that $\AAA$ is in $L^1_s(\Omega, \M, \mu,
    \mathcalb B(X))$. However, $\| \AAA_{t, \lambda} \| = 1$ for all $(t, \lambda) \in \Omega$ and therefore $\AAA$ does not have $\sigma$-finite support. $\hfill\triangle$
\end{example}
Note that $\| \AAA \|_1 = 0$ in the above example, compare with Corollary \ref{sigfinl1}.

The next result is a basic property of $L^p_s(\Omega, \M, \mu, \mathcalb B(X, Y))$ spaces.  

\begin{theorem}\label{CGT}
Let $1 \leqslant p<+\iii$ and assume that the function $\AAA \colon \WWW \rightarrow \mathcalb B(X, Y)$ is strongly $p$-integrable. Then we have  
$$\sup_{\| x \|_X=1} \int_\WWW \|\AAA_t x \|_Y^p\,d\mu(t) < +\infty.$$
\end{theorem}
A sketch of proof of the above theorem, as well as formulation of Proposition \ref{pBnormProp}, can be found in \cite[Page 64]{ABS}. For reader's convenience we present a proof, based on Closed Graph Theorem.
\begin{proof}[Proof of Theorem \ref{CGT}]  A natural (and  obviously linear) map $T : X \to L^p(\Omega, \mu, Y)$ is given by $Tx = \AAA x$. 
Let us assume that $x_n \to x$ in $X$ and $Tx_n \to (\varphi_t)_{t \in \Omega}$ in $L^p(\Omega, \mu, Y)$.
Since convergence in the norm of $L^p(\Omega, \mu, Y)$ implies existence of a subsequence converging a.\,e. pointwise we have $\lim_{k\to \infty} \AAA_t x_{n_k} = \varphi_t$ for $\mu$-almost all $t \in \Omega$. 
However, continuity of each $\AAA_t$ gives $\lim_{k\to \infty} \AAA_t x_{n_k} = \AAA_t x$ and therefore 
$Tx = (\varphi_t)_{t \in \Omega}$. This proves boundedness of $T$ and completes the proof.
\end{proof}
Theorem \ref{CGT} enables us to define a norm on $L^p_s(\WWW,\mu,\mathcalb B(X, Y))$, in the following  $L^p_s(\WWW,\mu,\mathcalb B(X, Y))$ is considered as a normed space with norm defined in \eqref{pBnorm} below.

\begin{proposition}\label{pBnormProp}
     $L^p_s(\WWW,\mu,\mathcalb B(X, Y))$ is a vector space for every $1 \leqslant p < +\iii$ and a norm on $L^p_s(\WWW,\mu,\mathcalb B(X, Y))$ is defined by the following formula:
    \begin{equation}\label{pBnorm}
        \|\AAA\|_p = \sup_{\|x\|_X=1} \sqrt[p]{\int_\WWW\|\AAA_t x\|_Y^p\,d\mu(t)}.
    \end{equation} 
\end{proposition}
\begin{proof}
Homogeneity and triangle inequality are obvious, non-degeneracy follows from pointwise identification of families in $L^p_s(\WWW,\mu,\mathcalb B(X, Y))$. 
\end{proof}

Using \eqref{Ulazakx}, one derives the following standard estimate:
\begin{multline}\label{normEst}
    \left\| \int_E \AAA \, d\mu \right\|  =\sup_{\|x\|_X=1}\left\|\left( \int_E \AAA \, d\mu\right)\!x \right\|_Y =\sup_{\|x\|_X=1} \left\|\prescript{\mathrlap{B}}{}{\int_{E}} \AAA x \,d\mu \right\|_Y \\
    \leqslant \sup_{\|x\|_X=1} \int_E\|\AAA x \|_Y\,d\mu \leqslant \| \AAA \|_1, \qquad E \in \M, \quad \AAA \in L^1_s(\Omega, \mu, \BXY).
\end{multline}

Let, for $\AAA \in L^1_s(\Omega, \M, \mu, \BXY)$, $\AAA_\M = \{ \int_E \AAA \, d\mu : E \in \M \}$ and set 
$$KL^1_s(\Omega, \M, \mu, \BXY) = \{\AAA \in L^1_s(\Omega, \mu, \BXY) : \AAA_\M \; \mbox{\rm is relatively compact}\}. $$
Using characterization of compactness in metric spaces by finite $\varepsilon$-nets and \eqref{normEst} one obtains the following proposition.
\begin{proposition}\label{ClSub}
    $KL^1_s(\Omega, \M, \mu, \BXY)$ is a subspace of $L^1_s(\Omega, \M, \mu, \BXY)$ which contains all simple functions.
\end{proposition}

In general, the space $L^p_s(\Omega, \mu, \BXY)$ does not need to be complete, 
but in the case of a discrete measure we have the following result.

\begin{proposition}
    \label{discr}
   The space $\ell_s^p(\mathcalb B(X, Y))$ is a Banach space for all $1 \leqslant p < +\infty$. 
\end{proposition}
\begin{proof}
Let $\left(\AAA^{(n)}\right)_{n=1}^\infty$ be sequence in $\ell_s^p(\mathcalb B(X, Y))$ such that
$$\sum_{n=1}^{\infty}\sup_{\|x\|_X=1}\sqrt[p]{\sum_{m=1}^{\infty} \left\|\AAA_m^{(n)}x\right\|_Y^p}=\sum_{n=1}^{\infty}\left\|\AAA^{(n)}\right\|_{p}=M<+\infty.$$
Thus, for every unit vector $x$ in $X$ we get, using Minkowski's inequality
$$\sqrt[p]{\sum_{m=1}^{\infty}\left(\sum_{n=1}^{\infty}\left\|\AAA_m^{(n)}x\right\|_Y\right)^p}\leqslant\sum_{n=1}^{\infty}\sqrt[p]{\sum_{m=1}^{\infty}\left\|\AAA_m^{(n)}x\right\|_Y^p}\leqslant M.$$
Thus, $\sum_{n=1}^\infty \AAA_m^{(n)}x$ converges in $X$ for every $m \in \N$. Hence, for all $x \in X$ and $m \in \N$ we can define $\AAA_m x:=\sum_{n=1}^{\infty}\AAA_m^{(n)}x$. 
We defined $\AAA_m \in \mathcalb B(X, Y)$ with $\|\AAA_m\| \leqslant M$.
Set $\AAA=(\AAA_m)_{m\in\N}$. Then, using Minkowski's inequality, we have 
\begin{equation*}
    \begin{split}
   &\left\| \AAA - \sum_{n=1}^N \AAA^{(n)} \right\|_p^p = \sup_{\|x\|_X=1} \sum_{m=1}^{\infty} \left\| \AAA_mx -\sum_{n=1}^N \AAA_m^{(n)}x \right\|_Y^p \\
   & = \sup_{\|x\|_X=1} \sum_{m=1}^{\infty} \left\| \sum_{n=N+1}^{\infty} \AAA_m^{(n)}x \right\|_Y^p \leqslant \sup_{\|x\|_X=1} \sum_{m=1}^{\infty} \left(\sum_{n=N+1}^{\infty} \left\| \AAA_m^{(n)}x \right\|_Y \right)^p \\
   & \leqslant \left( \sum_{n=N+1}^{\infty}\sup_{\|x\|_X = 1} \sqrt[p]{\sum_{m=1}^{\infty} \left\| \AAA_m^{(n)}x\right\|_Y^p} \right)^p = \left( \sum_{n=N+1}^{\infty} \left\| \AAA^{(n)}\right\|_p \right)^p, \qquad N \in \N.
    \end{split}
\end{equation*}
Taking $N\to\infty$ we get desired conclusion.
\end{proof}
Let us point out that this completeness does not follow from coincidence of $\ell^p_s(\BXY)$ with the corresponding Bochner space $l^p(\BXY)$, which in general does not hold, see Example \ref{NeZatvorenost}. Proposition \ref{BochnerFinite} below gives another sufficient condition for completeness of $L^p_s(\Omega, \M, \mu, \BXY)$.
\begin{lemma}\label{merljivostTenzora}
    Let $f:\Omega\to Y$ be strongly $\mu$-measurable and $x^*\in X^*$. Then the mapping $\Omega\ni t\mapsto f(t) \otimes x^* \in \BXY$ is strongly $\mu$-measurable 
    function. 
\end{lemma}
\begin{proof}
    Let $T_{x^*}\colon Y\rightarrow\BXY$ be defined by $T_{x^*}y = y \otimes x^*$, for all $y\in Y$. Then, $T_{x^*}$ is linear and, since $\|T_{x^*}y\| = \| y \otimes x^* \| = \|x^*\|_{X^*}\cdot\|y\|_Y$, continuous with $\|T_{x^*}\| = \|x^*\|_{X^*}$. 
    Let $\AAA\colon\WWW\rightarrow\BXY$ be defined by $\AAA_t = f(t) \otimes x^* = (T_{x^*}\circ f)(t)$, for $t\in\WWW$. Finally, as $T_{x^*}$ is measurable (being continuous) with $T_{x^*}(0)=0$ and $f$ is strongly $\mu$-measurable, it follows that $\AAA=T_{x^*}\circ f$ is strongly $\mu$-measurable due to \cite[Corrolary~1.1.24]{ABS}.
\end{proof}

\begin{proposition}\label{BochnerFinite}
    Let $\dim X=n$. Then every $\AAA \in \mathcal M_s(\Omega, \M, \mu, \BXY)$ is strongly $\mu$-measurable. Further,  let $1\leqslant p<+\iii$. Then $ L^p_s(\WWW,\M,\mu,\mathcalb{B}(X,Y))=L^p(\WWW,\M,\mu,\mathcalb{B}(X,Y))$ as vector spaces. Moreover, 
    $$\|\AAA\|_p\leqslant\|\AAA\|_{L^p}\leqslant n\cdot\|\AAA\|_p, \qquad \AAA\in L^p_s(\WWW,\M,\mu,\mathcalb{B}(X,Y)),$$
 implying completeness of $L^p_s(\WWW,\M,\mu,\mathcalb{B}(X,Y))$.
  
\end{proposition}
\begin{proof}
As $X$ is $n$-dimensional, due to Auerbach's Lemma (for details see \cite[Proposition~6.26]{DJT}) there exist basis of unit vectors $e_1,e_2,\ldots,e_n\in X$ and basis of unit functionals $e_1^*,e_2^*,\ldots,e_n^*\in X^*$ such that $\scal{e_i}{e_j^*}=\delta_{ij}$, for all $i,j$ in $\{1,2,\ldots,n\}$. Let $\AAA\in \mathcal M_s(\WWW,\M,\mu,\BXY)$ be arbitrary. Since for every $x\in X$ holds $x=\sum_{k=1}^n\scal{x}{e_k^*}e_k$, it follows that 
    \begin{equation*}
        \AAA_t=\sum_{k=1}^n (\AAA_t e_k) \otimes  e_k^*, \qquad t \in \Omega.
    \end{equation*}
    As all of the mappings $\AAA e_k$ are strongly $\mu$-measurable $Y$-valued functions, by Lemma \ref{merljivostTenzora} the mappings $(\AAA e_k) \otimes e_k^*$ are strongly $\mu$-measurable $\BXY$-valued functions implying that $\AAA$ is also such. 
    
    Now let us assume $\AAA\in L^p_s(\WWW,\M,\mu,\BXY)$. Since for all $t\in\Omega$ we have \begin{equation*}
        \|\AAA_t\| \leqslant \sum_{k=1}^n \| (\AAA_t e_k) \otimes  e_k^* \| = \sum_{k=1}^n \|\AAA_te_k\|_Y \cdot \|e_k^*\|_{X^*} = \sum_{k=1}^n \|\AAA_t e_k\|_Y\end{equation*}
due to Holder inequality we get that for all $t\in\WWW$ we have
\begin{equation}\label{helder}
    \|\AAA_t\|^p \leqslant \left( \left( \sum_{k=1}^n \|\AAA_te_k\|_Y^p \right)^\frac{1}{p} \left( \sum_{k=1}^n 1^\frac{p}{p-1} \right)^\frac{p-1}{p} \right)^p = n^{p-1} \cdot \sum_{k=1}^n\|\AAA_te_k\|_Y^p.
\end{equation}
As $X$ is separable (being finite-dimensional), $\Omega\ni t\mapsto\|\AAA_t\|\in\R$ is a $\mu$-measurable function. Moreover, the function $\AAA e_k$ belongs to the Banach space $ L^p(\Omega, \M, \mu,Y)$ for all $k\in\{1,\ldots,n\}$. Then, using \eqref{helder}, we conclude that $\AAA\in L^p(\WWW,\mu,\BXY)$. By integrating \eqref{helder}, we get
\begin{equation*}
\begin{split}
\|\AAA\|_{L^p}^p & \leqslant n^{p-1} \sum_{k=1}^n \int_\WWW \|\AAA_te_k\|_Y^p\,d\mu(t)\\& \leqslant n^{p-1}\sum_{k=1}^n\sup_{\|x\|_X\leqslant1}\int_\WWW\|\AAA_tx\|_Y^p\,d\mu(t)  = n^p\cdot\|\AAA\|_p^p,
\end{split}
\end{equation*}
proving the proclaimed equivalence of norms. Since $L^p(\WWW,\M,\mu,\BXY)$ is a Banach space, $L^p_s(\WWW,\M,\mu,\BXY)$ is also a Banach space.
\end{proof}

\begin{remark}
We present an example where the constant $n$ is attained. We take $X=Y=\C^n$ with $\ell^1$-norm, and define $\AAA\in\ell^1(\BXY)$ by $\AAA_k=e_k\otimes e_k^*$ for $1 \leqslant k \leqslant n$, where $\{e_1,e_2,\ldots,e_n\}$ is canonical basis of $\C^n$. Direct calculation gives us (for more details see Example \ref{BandpB}) where $\|\AAA\|_1=1$ and $\|\AAA\|_{L^1}=n=n\cdot\|\AAA\|_1$.
\end{remark}

\begin{lemma}\label{Complemented}
Let $M$ be a complemented subspace of $X$, with complement $N$, and let $\pi\colon X \to M$ be the corresponding continuous projection along $N$. Assume $\AAA = (\AAA_t)_{t \in \Omega}$ is a strongly $\mu$-measurable family in $\mathcalb B(M, Y)$. Then $\AAA' = (\AAA_t \pi)_{t \in \Omega}$ is a strongly $\mu$-measurable family in $\BXY$.
    
If, moreover, $\AAA \in L^p(\Omega, \M, \mu, \mathcalb B(M, Y))$, then $\AAA' \in L^p(\Omega, \M, \mu, \BXY)$. 
\end{lemma}
\begin{proof}
Let $s_n : \Omega \to \mathcalb B(M, Y)$ be a sequence of $\mu$-simple functions such that $\lim_{n\to\infty} \| \AAA_t - s_n(t) \| = 0$ $\mu$-almost everywhere. Then the sequence $s_n'(t) = s_n(t)\pi$ is a sequence of $\mu$-simple functions in $\BXY$ such that $\lim_{n\to\infty} \| \AAA'_t - s_n'(t) \| = 0$ $\mu$-almost everywhere, which proves the first statement. The second one follows from $\| \AAA_t' \| \leqslant \| \pi \| \cdot \| \AAA_t \|$. 
\end{proof}

\subsection{Examples and counterexamples}

Clearly $L^p(\WWW,\mu,\mathcalb B(X, Y))$ is contained in $L^p_s(\WWW,\mu,\mathcalb B(X, Y))$ and 
$$\| \AAA \|_p \leqslant \| \AAA \|_{L^p(\Omega, \mu, \mathcalb B(X, Y))}, \qquad 
\AAA \in L^p(\WWW,\mu,\mathcalb B(X, Y)).$$

Our first task is to compare the space $L^p_s(\WWW,\mu,\mathcalb B(X, Y))$ to the narrower Bochner space $L^p(\WWW,\mu,\mathcalb B(X, Y))$, this is done by the following examples.
The first one appeared in \cite{MAMK} (in a special case $p=2$) as an Example 2.12, it was noted in \cite{MAMK} that $\AAA$ is pointwise Dunford integrable, but not strongly $\mu$-measurable. 
However, more is true, it is in fact strongly $p$-integrable for every $1 \leqslant p < +\infty$.

\begin{example}\label{pBnijeB}
    Let $\Omega = [0, 1]$, $X = L^p((0,1),m)$ where $1 \leqslant p < +\infty$ and let $\AAA_t$ be the multiplication operator by $\chi_{(0, t)}$ for all $0 \leqslant t \leqslant 1$. 
    It is easily seen that the range of $\AAA$ is a non-separable subset of $\mathcalb B(X)$ and therefore $\AAA$ is not strongly $\mu$-measurable and thus does not belong to $L^p(\Omega, m, \mathcalb B(X))$. 
    However, for every $f \in X$ the mapping $[0,1]\ni t \mapsto \chi_{(0,t)} f \in X$ is a continuous  $X$-valued function on $[0, 1]$ and thus Bochner $p$-integrable.\hfill$\triangle$
\end{example}

\begin{example}\label{NeZatvorenost}
    Let $1 \leqslant p, q < +\infty$, set $X = \ell^q$ and define $\AAA : \N \to \mathcalb B(X)$ by $\AAA_n x = \frac{1}{\sqrt[p]{n}} x_n \cdot e_1$ for $x = (x_n)_{n=1}^\infty \in \ell^q$. 
    Then for every $x = (x_n)_{n=1}^\infty \in \ell^q$ we have
    $$\sum_{n=1}^{\infty} \| \AAA_n x \|^p = \sum_{n=1}^{\infty} \frac{1}{n} \cdot | x_n |^p \leqslant \left( \sum_{n=1}^{\infty} \frac{1}{n^{\frac{r}{r-1}}} \right)^{\frac{r-1}{r}}\sqrt[r]{\sum_{n=1}^{+\iii} |x_n|^{pr}} < +\iii,$$
    where $r > 1$ is selected so that $pr \geqslant q$. Hence $\AAA \in \ell^p_s(\mathcalb B(X))$. 
    Since $\| \AAA_n \| = \frac{1}{\sqrt[p]{n}}$ we see that $\AAA$ is not in $\ell^p(\mathcalb B(X))$. 
    Let $q'$ be the exponent conjugate to $q$. Then $\AAA_n^* \in \mathcalb B(\ell^{q'})$ is given by $\AAA_n^* x = \frac{1}{\sqrt[p]{n}} x_1 e_n$, $x = (x_n)_{n=1}^\infty \in \ell^{q'}$. However, for $x_1 \not= 0$, the series $\sum_{n=1}^\infty \| \AAA_n x \|^p$ is divergent and thus $\AAA^*$ is not in $\ell^p_s(\mathcalb B(\ell^{q'}))$. \hfill$\triangle$
\end{example}

In Example \ref{pBnijeB} non-separability of the range excluded $\AAA$ from the space $L^p(\Omega, m, \mathcalb B(X))$, although $\int_{(0,1)} \|\AAA_t\|^p\,dm(t) = 1 < +\iii$.
In Example \ref{NeZatvorenost} it was the size of norm of $\AAA$ that did the trick. 
The same Example \ref{NeZatvorenost} shows that $\AAA \in L^p_s (\Omega, \mu, \mathcalb B(X,Y))$ does not imply $\AAA^\ast \in L^p_s (\Omega, \mu, \mathcalb B(Y^*, X^*))$, even in the case when $X$ is a Hilbert space (take $q=2$ in Example \ref{NeZatvorenost}). 

Hence we conclude: spaces $L^1(\Omega, \mu, \BH)$ of Bochner integrable functions and $\GG(\Omega, \mu, \BH)$ of Gelfand integrable functions are invariant under conjugation, but the intermediate space $L^1_s(\Omega, \mu, \BH)$ in general is not invariant.
However, certain type of integrability of $\AAA^*$ is present for $\AAA \in L^1_s(\WWW,\M,\mu,\BXY)$. Namely, due to duality $\mathcalb{B}(Y^*,X^*)\cong(Y^*\widehat{\otimes}_\pi X)^*$ we can state the following result.

\begin{lemma}\label{adjungovani}
    If $\AAA\in L^1_s(\WWW,\M,\mu,\BXY)$, then $\AAA^*\in\GG(\WWW,\mu,\mathcalb{B}(Y^*,X^*))$ and $$\left(\int_E\AAA \,d\mu\right)^*=\int_E\AAA^*\,d\mu, \qquad E \in \M.$$ 
\end{lemma}
\begin{proof}
    The mapping $\Omega\ni t\mapsto\scal{x}{\AAA_t^*y^*}=\scal{\AAA_tx}{y^*}\in\C$ is $\mu$-measurable for all $x\in X$ and $y^*\in Y^*$ and
\begin{equation*}
     \int_\WWW|\scal{x}{\AAA_t^*y^*}|\,d\mu(t)=\int_\WWW|\scal{\AAA_tx}{y^*}|\,d\mu(t)\leqslant\|x\|_X\|y^*\|_{Y^*}\|\AAA\|_1<+\iii.
   \end{equation*}
   Hence, by \cite[Lemma~1.8]{MS}, $\AAA^*\in\GG(\WWW,\M,\mu,\mathcalb{B}(Y^*,X^*))$. Moreover, for $E \in \M$:
  \begin{equation*}
  \begin{split}
      \scal{x}{\left(\int_E\AAA^*d\mu\right)y^*}&=\int_E\scal{x}{\AAA_t^*y^*}d\mu(t)\\&=\int_E\scal{\AAA_tx}{y^*}d\mu(t)=\scal{\left(\int_E\AAA d\mu\right)x}{y^*}. \qedhere
  \end{split}
  \end{equation*}
\end{proof}

In the next example we show that the norms $\|\AAA\|_{L^p} = (\int_\Omega\|\AAA_t\|^p\,d\mu(t))^{1/p}$ and 
$\|\AAA\|_p$ in general are not equivalent on $L^p(\Omega,\mu,\mathcalb B(X))$.

\begin{example}\label{BandpB}
    Let $1 \leqslant p, q < \infty$, set $X = \ell^q$, and let $(e_n)_{n=1}^\infty$ be the standard Schauder basis in $\ell^q$. For every $n\in\N$ define an o.v. function $\AAA^{(n)} \colon \N \rightarrow \mathcalb{B}(X)$ by $\AAA^{(n)}_m = \chi_{\{1,\ldots,n\}}(m)e_m\otimes e_m^*$. 
    We have $\| \AAA^{(n)}_m \| = 1$ for $1 \leqslant m \leqslant n$ and therefore $\|\AAA^{(n)}\|_{L^p} = n^{1/p}$. Let us fix $x = (\xi_n)_{n=1}^\infty$ in $\ell^q$. Then $\AAA^{(n)}x : \N \to \ell^q$ is given by
    $\AAA^{(n)}x (k) = \xi_k e_k$ for $1 \leqslant k \leqslant n$ and $\AAA^{(n)}x (k) = 0$ for $k > n$. 
    If $q \leqslant p$, then 
    $$\sum_{k=1}^\infty\| \AAA^{(n)}x(k) \|^p_{\ell^q} = \sum_{k=1}^n |\xi_k |^p \leqslant \| x \|^p_{\ell^p} \leqslant \| x \|^p_{\ell^q}.$$ 
    Thus $\| \AAA^{(n)} \|_p \leqslant 1$, in fact taking $x = e_1$ it is easily seen that $\| \AAA^{(n)} \|_p = 1$. Therefore, in the case $q \leqslant p$, the two norms are not equivalent.
    
    In the remaining case $p < q$ we have, for $x = (\xi_n)_{n=1}^\infty$ in $\ell^q$:
    $$\sum_{k=1}^\infty\| \AAA^{(n)}x (k) \|^p_{\ell^q} = \sum_{k=1}^n | \xi_k |^p \leqslant n^{1-p/q}
    \left(\sum_{k=1}^n | \xi_k |^q \right)^{p/q} \leqslant n^{1-p/q} \| x \|_{\ell^q}^p.$$
    Hence $\| \AAA^{(n)} \|_p \leqslant n^{1/p-1/q}$ and again the two norms are not equivalent.
    $\hfill\triangle$
\end{example}

The next example shows that the norms $\|\cdot\|_{\mathbb{G}}$ and $\|\cdot\|_{1}$ are not equivalent on $L^1_{s}(\Omega,\mu,\BH)$, in the general case.

\begin{example}
Let $\Omega=\N$ and let $(e_n)_{n\in\N}$ be an o.\,n.\,b. for $\HH$. We set
$$\AAA_n^{(m)}=\frac{1}{n}\chi_{\{1,\ldots,m\}}(n)\cdot e_n\otimes e_1^*, \qquad m, n \in \N.$$
Then, for the family $\AAA^{(m)}=\left(\AAA_n^{(m)}\right)_{n\in\N}$ we have:
$$\left\|\AAA^{(m)}\right\|_{1}=\sup_{\|f\|_\HH=1}\int_\N\left\|\AAA^{(m)}_n f\right\|_\HH\,d\mu(n) = \sup_{\|f\|_\HH=1}\sum_{n=1}^m\frac{1}{n}|\scal{f}{e_1}|\cdot\|e_n\|_\HH = \sum_{n=1}^m\frac{1}{n},$$
\begin{equation*}
    \begin{split}
\left\|\AAA^{(m)}\right\|_{\mathbb{G}}&=\sup_{\|f\|_\HH=\|g\|_\HH=1} \int_\N|\scal{\AAA^{(m)}_nf}{g}|\,d\mu(n) = \sup_{\|f\|_\HH=\|g\|_\HH=1} \sum_{n=1}^m\frac{1}{n}\cdot|\scal{f}{e_1}\scal{e_n}{g}|\\
&\leqslant\sup_{\|g\|_\HH=1}\sqrt{\sum_{n=1}^m\frac{1}{n^2}}\cdot\sqrt{\sum_{n=1}^m|\scal{g}{e_n}|^2} = \frac{\pi}{\sqrt{6}}.
    \end{split}
\end{equation*}
Since the harmonic series is divergent, these two norms are not equivalent.
$\hfill\triangle$
\end{example}

\section{Main results on $L^p_s(\Omega, \M, \mu, \mathcalb B(X, Y))$ spaces}\label{MAIN}

In this section we present key results, Theorems \ref{ACmuPb}, \ref{UCA} and \ref{vecmes}. Main point is that a family $\AAA$ in the space $ L_s^1(\Omega, \M, \mu, \BXY)$ generates countably additive $\BXY$-valued measure provided $X^*$ does not contain an isomorphic copy of $\mathfrak{c}_0$. This is achieved through intermediary results, among them Theorem \ref{l1sbuk} has independent interest.
The following proposition will be superseded by Theorem \ref{l1sbuk} below and used in its proof.

\begin{proposition}\label{limnula}
    Assume $X^*$ does not contain an isomorphic copy of $\mathfrak{c}_0$ and let $\AAA\in\ell_s^1(\mathcalb B(X, Y))$. Then $\lim_{n\to\infty} \| \AAA_n \| = 0$.
\end{proposition}
\begin{proof}
We will prove the statement by reducing it to a special case of rank-$1$ operators. Let 
$\AAA\in \ell^1_s(\mathcalb B(X, Y))$ and let $y_n^*$ be unit vectors $Y^*$ such that $\|\AAA_n^* y^*_n\|_{X^*} \geqslant \frac{1}{2}\|\AAA_n^*\| = \frac{1}{2}\|\AAA_n\|$. Next, we choose a unit vector $y_0 \in Y$ and set, for every $n \in \N$, $\BBB_n = (y_0 \otimes y^*_n) \AAA_n$. First, we have $$\|\BBB_n x\|_Y \leqslant \|y_0 \otimes y^*_n\| \cdot \|\AAA_n x\|_Y = \|\AAA_n x\|_Y$$while, due to equality \eqref{tesor}, we have $\BBB_n = y_0\otimes \AAA_n^*y_n^*$ implying $$\|\BBB_n x\|_Y = |(\AAA_n^* y_n^*)(x)|\cdot \|y_0\|_Y = |(\AAA_n^* y_n^*) (x)|.$$ Combining these two facts, we get
\begin{equation}\label{l1wl1s}
    \sum_{n=1}^\iii |(\AAA_n^* y_n^*) (x)| = \sum_{n=1}^\iii \|\BBB_n x\|_Y \leqslant \sum_{n=1}^\infty \|\AAA_n x\|_Y <+\infty,
\end{equation}
for every $x\in X,$ as $\AAA\in \ell^1_s(\mathcalb B(X, Y))$. By taking supremum over $B_X$ in the \eqref{l1wl1s} we get that $\|(\AAA^*_ny^*_n)_{n=1}^\iii\|_{\ell^1_w(X^*)} \leqslant \|\AAA\|_1$. Hence, by Lemma \ref{ell1slslz}, the sequence $(\AAA_n^* y_n^*)_{n=1}^\iii$ belongs to $\ell_w^1 (X^*)$. Finally, due to Bessaga-Pelczynski Theorem (see \cite[Corollary 5, Page 22]{DU}), as $X^*$ does not contain isomorphic copy of $\mathfrak{c}_0$, the series $\sum_{n=1}^\infty \AAA_n^* y_n^*$ converges unconditionally. Hence $\lim_{n\to\infty}\|\AAA_n^* y_n^*\|_{X^*} = 0$ and by the choice of vectors $y_n^*$ we get $\lim_{n\to\infty}\|\AAA_n\|= 0$, as proclaimed.
\end{proof}

The next result shows that the passage from pointwise summability on vectors to unconditional convergence in the operator norm is governed precisely by the geometry of the dual space $X^*$. 
\begin{theorem}\label{l1sbuk} 
The series $\sum_{n=1}^\infty \AAA_n$ converges unconditionally in $\BXY$ for all $\AAA\in\ell^1_s(\BXY)$ if and only if $X^*$ does not contain isomorphic copy of $\mathfrak{c}_0$. 
\end{theorem}
\begin{proof}
If $X^*$ contains isomorphic copy of $\mathfrak{c}_0$ then by Bessaga-Pelczynski Theorem 
there exists a series $\sum_{n=1}^\infty x_n^*$ in $X^*$ which is not unconditionally norm convergent but $\sum_{n=1}^\infty | \langle x_n^*, x^{**} \rangle | < +\infty$ for all $x^{**} \in X^{**}$. 
Let us fix a nonzero vector $y_0$ in $Y$ and set $\AAA_n = y_0 \otimes x_n^* \in \mathcalb B(X, Y)$ for $n \in \N$.
Since 
$$\sum_{n=1}^\infty \|\AAA_n x\|_Y = \sum_{n=1}^\infty \| \langle x, x_n^* \rangle y_0 \|_Y = \| y_0 \|_Y \cdot \sum_{n=1}^\infty |x_n^*(x)| < +\infty,$$
we conclude that $\AAA\in\ell^1_s(\mathcalb B(X, Y))$. Moreover, 
$(\AAA_{\pi(n)})_{n=1}^\infty$ belongs to $\ell^1_s(\mathcalb B(X, Y))$ for every permutation $\pi : \N \to \N$. Since 
$$\sum_{k=1}^n \AAA_{\pi(k)} = y_0 \otimes \sum_{k=1}^n x_{\pi(k)}^*,$$ 
unconditional norm convergence of $\sum_{n=1}^\infty \AAA_n$ would imply unconditional norm convergence of $\sum_{n=1}^\infty x_n^*$, which is not true.

Now we consider the converse, i.\,e. we assume $X^*$ contains no isomorphic copy of $\mathfrak{c}_0$.
Since for any permutation $\pi$ of $\N$ the family 
$(\AAA_{\pi(n)})_{n=1}^\iii$ is also in $\ell^1_s(\mathcalb B(X, Y))$, it suffices to prove that the series $\sum_{n=1}^{\infty}\AAA_n$ converges in the norm of $\mathcalb B(X, Y)$. Assume, to the contrary, that there exist $\varepsilon>0$ and strictly increasing sequences $(a_k)_{k=1}^\iii$ and $(b_k)_{k=1}^\iii$ of natural numbers such that 
$$a_k\leqslant b_k,\,\,b_k<a_{k+1}\,\,\,\text{and}\,\,\, \left\|\sum\limits_{n=a_k}^{b_k}\AAA_{n}\right\|\geqslant \varepsilon\,\,\,\text{for all}\,\,\,k\in\N.$$ 
Define $\CCC_k = \sum_{n=a_k}^{b_k}\AAA_{n}$ for $k \in \N$.
Then, $\CCC = (\CCC_k)_{k=1}^\infty \in \ell_s^1(\mathcalb B(X, Y))$ since
$$\sum_{k=1}^{\infty}\|\CCC_k x\|_Y = \sum_{k=1}^{\infty} \left\| \sum_{n=a_k}^{b_k}\AAA_n x \right\|_Y \leqslant\sum_{k=1}^{\infty} \sum_{n=a_k}^{b_k} \|\AAA_n x\|_Y \leqslant \sum_{n=1}^{\infty} \|\AAA_n x \|_Y < +\infty$$
for all $x\in X$.
Then, by Proposition \ref{limnula}, $\lim_{k\to\infty}\|\CCC_k\|= 0$
which is impossible since $\|\CCC_k\|\geqslant\varepsilon>0$, $k\in\N$. Thus, the series $\sum_{n=1}^{\iii} \AAA_n$ is norm convergent in $\mathcalb B(X, Y)$.
\end{proof} 

Since a Hilbert space does not contain an isomorphic copy of $\mathfrak{c}_0$ we obtain
\begin{corollary}
For every $\AAA\in\ell_s^1(\BH)$ the series $\sum_{n=1}^\infty\AAA_n$ is unconditionally convergent in the norm of $\BH$.
\end{corollary}

Under an additional assumption on $X$, Proposition \ref{limnula} can be further refined. More precisely, the next theorem is an extension of Theorem \ref{kotipniz} to sequences of operators.

\begin{theorem}
Assume $X^*$ has cotype $q<+\infty$, 
and let $\AAA\in \ell^1_s (\mathcalb B(X, Y))$. Then $\AAA\in\ell^q (\mathcalb B(X,Y))$ and we have the following estimate, where $C_q = C_q(X^*)$:
\begin{equation}\label{cotype}
\left( \sum_{n=1}^\infty \| \AAA_n \|^q \right)^{1/q} \leqslant  C_q \cdot \left\| \AAA \right\|_1.
\end{equation}
\end{theorem}


\begin{proof}
Let $\AAA\in\ell^1_s(\mathcalb B(X, Y))$ and $\varepsilon > 0$. We can choose unit vectors $y_n^*\in Y^*$ such that $\|\AAA_n^* y_n^*\|_{X^*} \geqslant \frac{1}{1+\varepsilon}\|\AAA_n\|$. As in the proof of Proposition \ref{limnula} select a unit vector $y_0$ in $Y$ and define operators 
$$\BBB_n = (y_0 \otimes y_n^*) \AAA_n = y_0 \otimes (\AAA_n^* y_n^*)\colon X \to Y.$$
Then $\BBB\in \ell^1_s (\mathcalb B(X,Y))$ with
$\| \BBB\|_1 \leqslant \| \AAA\|_1$. Moreover, using \eqref{l1wl1s} we again get $(\AAA_n^*y_n^*)_{n=1}^\iii \in \ell^1_w (X^*)$ where $\| (A_n^*y_n^*)_{n=1}^\iii \|_{\ell_w^1(X^*)} =  \|\BBB\|_1$. Due to Theorem \ref{kotipniz} we have $(\AAA_n^*y_n^*)_{n=1}^\iii\in\ell^q(X^*)$ and
\begin{equation}\label{comb2}
\left(\sum_{n=1}^\iii\|\AAA_n^*y_n^*\|^q_{X^*}\right)^{1/q}\leqslant C_q\cdot\|(\AAA_n^*y_n^*)_{n=1}^\iii\|_{\ell^1_w(X^*)}=C_q\cdot\|\BBB\|_1\leqslant C_q\cdot\|\AAA\|._1
\end{equation}
Now, as $\|\AAA_n^* y_n^*\| \geqslant \frac{1}{1+\varepsilon}\|\AAA_n\|$ for all $n\in\N$, by the \eqref{comb2} we  get
$$\frac{1}{1+\varepsilon} \left(\sum_{n=1}^\iii \|\AAA_n\|^q\right)^{1/q} \leqslant \left(\sum_{n=1}^\iii \|\AAA_n^* y_n^*\|_{X^*}^q\right)^{1/q} \leqslant C_q \cdot\left\|\AAA\right\|_1.$$
Finally, by letting $\varepsilon\to 0+$ we obtain \eqref{cotype}, completing the proof.
\end{proof}
\begin{remark}
    The previous theorem states that the inclusion $$j\colon\ell^1_s(\BXY)\rightarrow\ell^q(\BXY)$$ is continuous when $X^*$ has a finite cotype $q$. Moreover, due to Remark \ref{kotiptac} we can further estimate $\|j\|\leqslant \pi_{q,1}(I_X)$. Actually, we have $\|j\| = \pi_{q,1}(I_X)$. To see previous equality, for $\varepsilon > 0$ we can find sequence $(x_n^*)_{n=1}^\iii\in\ell^1_w(X^*)$ such that $$\|(x_n^*)_{n=1}^\iii\|_{\ell^1_w(X^*)} = 1,\qquad\left(\sum_{n=1}^\iii\|x_n^*\|^q_{X^*}\right)^{1/q}\geqslant \pi_{q,1}(I_X) -\varepsilon.$$ Next, for all $n\in\N$, we define operators $\AAA_n = x_n^*\otimes y_0$ where $y_0\in Y$ is arbitrary unit vector. Then $\|\AAA_n\| = \|x_n^*\|_{X^*}$ and similarly to \eqref{l1wl1s} we have $\|\AAA\|_1 = \|(x_n^*)_{n=1}^\iii\|_{\ell^1_w(X^*)} = 1$. Therefore, we get inequality
    $$\left(\sum_{n=1}^\iii \|\AAA_n\|^q\right)^{1/q} \geqslant (\pi_{q,1}(I_X) -\varepsilon) \cdot\|\AAA\|_1,$$
    implying that $\|j\|\geqslant\pi_{q,1}(I_X).$
\end{remark}

In the special case $X = Y= \HH$, due to Theorem \ref{Pi21}, the previous Theorem states that $\ell^1_s(\BH)$ is continuously embedded in $\ell^2(\BH)$. More precisely:

\begin{corollary}\label{l1sl2} 
If $\AAA\in\ell^1_s(\BH)$ then we have $$\AAA\in\ell^2(\BH)\,\,\,\text{and}\,\,\,\|\AAA\|_{\ell^2(\BH)}\leqslant\|\AAA\|_1.$$ 
\end{corollary}

One expects that exponent $2$ is optimal in Corollary \ref{l1sl2}, as cotype of a space is always bounded below by $2$. This is confirmed by the following modification of Example \ref{NeZatvorenost}. 

\begin{example}
Define $\AAA\colon\N\rightarrow\BH$ by the expression 
$$\AAA_n=\frac{1}{\sqrt{n}\ln(n+1)}\cdot e_n\otimes e_n^*, \qquad n \in \N.$$  
Then, for every $f\in\HH$ we have 
$$\sum_{n=1}^\iii\|\AAA_n f\|_\HH = \sum_{n=1}^\iii\frac{|\scal{f}{e_n}|}{\sqrt{n}\ln(n+1)}\leqslant\sqrt{\sum_{n=1}^\iii\frac{1}{n\ln^2(n+1)}}\cdot\|f\|_\HH < +\iii,$$
    implying that $\AAA\in\ell^1_s(\BH)$. On the other hand, for every $p\in[1,2)$ we have
$$\sum_{n=1}^\iii\|\AAA_n\|^p=\sum_{n=1}^\iii\frac{1}{n^\frac{p}{2}\ln^p(n+1)}=+\iii,$$
    since $\frac{p}{2}<1$, proving that $\AAA\notin\ell^p(\BH).$ \hfill$\triangle$
 \end{example}

\subsection{Induced vector measures}

By Proposition \ref{wkstcameasure}, every $\AAA$ in the space $\GG(\WWW,\M,\mu,\BH)$ induces a weakly$^*$ countably additive $\BH$-valued measure $\mu_\AAA$ defined by \eqref{meranu}. 
Also, $\mu_\AAA$ is countably additive if $\AAA\in L^p_s(\WWW,\M,\mu,\mathcalb{B}(X,Y))$, provided $\mu(\WWW)<+\iii$ and $p>1$, see \cite[Proposition~3.2]{BJN}. 
Moreover, the same is true for $p=1$ with the additional assumption of uniform integrability of the family $\{\AAA x:\|x\|_X\leqslant1\}$ in $L^1(\WWW,\M,\mu,Y)$, see \cite[Remark~3.3]{BJN}. 

Let us recall that a subset $K$ of $L^1(\WWW,\M,\mu,X)$ is called \textit{uniformly integrable} (see \cite[Page~101]{DU}) if 
$$\lim_{\mu(E)\rightarrow0}\sup_{f\in K}\int_E\|f\|_X\,d\mu=0,$$
i.e. if for every $\varepsilon>0$ there exists $\delta>0$ such that for every $\mu(E)<\delta$ and $f\in K$ holds $\int_E\|f\|d\mu<\varepsilon$. Note, that this is equivalent to uniform integrability of family of scalar functions $\{\|f\|:f\in K\}$.

Theorem \ref{ACmuPb} below shows that countable additivity of $\nu_\AAA$ can be concluded if the finiteness of measure $\mu$ and uniform integrability are replaced with a mild condition on $X^*$. 
Moreover, the condition of uniform integrability of the family $\{\AAA x : \| x \|_X \leqslant 1\}$ is apriori satisfied for $\AAA \in L^1_s(\Omega, \mu, \BXY)$ if $X$ does not contain isomorphic copy of $\ell^1$, see Corollary \ref{UINT} below.

\begin{theorem}\label{ACmuPb}
Assume $X^*$ does not contain an isomorphic copy of $\mathfrak{c}_0$. Then, for every element $\AAA$ in the space $ L_s^1(\Omega,\M,\mu,\BXY)$, the induced vector measure $\nu_\AAA:\M\to\BXY$ defined by the expression \begin{equation}\label{nuA}\nu_\AAA(E) = \int_{E}\AAA\,d\mu\,\,\,\text{for all}\,\,\, E\in\M\end{equation} is countably additive.
\end{theorem}
\begin{proof}
Let $(E_n)_{n=1}^\iii$ be arbitrary sequence of mutually disjoint sets in the $\sigma$-algebra $\M$ and let $E$ be the union of these sets. Now, for all $x\in X$, 
using \eqref{Ulazakx} we obtain
\begin{equation*}
\begin{split}
\sum_{n=1}^{\infty} \|\nu_\AAA(E_n) x \|_Y & = \sum_{n=1}^{\infty} \left\| \left( \int_{E_n} \AAA\,d\mu \right) x \right\|_Y = \sum_{n=1}^{\infty} \left\| \prescript{\mathrlap{B}}{}{\int_{E_n}} \AAA_t x \,d\mu(t) \right\|_Y\\ &
\leqslant \sumn\int_{E_n} \| \AAA_t x \|_Y\,d\mu(t) \leqslant \int_\WWW \| \AAA_t x \|_Y \,d\mu(t) <+\iii,
\end{split}
\end{equation*}
as $\AAA$ is strongly integrable.
Thus $(\nu_\AAA(E_n))_{n=1}^\iii\in \ell_s^1(\BXY)$, so by Theorem \ref{l1sbuk} the series $\sum_{n=1}^{\infty}\nu_\AAA(E_n)$ converges (unconditionally) in  norm of $\BXY$. 
However, the same series  converges strongly to the operator $\nu_\AAA (E)$ in $\BXY$
because $\AAA x\in L^1(\Omega,\mu,Y)$ for every $x\in X$. From the uniqueness of limit we get  
$$\nu_\AAA\left(\bigsqcup\limits_{n=1}^{\infty}E_n\right)=\sum\limits_{n=1}^{\infty}\nu_\AAA(E_n),$$
where the last series converges in the norm of $\BXY$.
\end{proof}

\begin{corollary}\label{AzvezdaMera}
Let $\AAA\in L^1_s(\WWW,\M,\mu,\BXY)$, where
$X^*$ does not contain $\mathfrak{c}_0$. Then, the vector measure $\nu_{\AAA^*}\colon\M\rightarrow\mathcalb{B}(Y^*,X^*)$ defined as in \eqref{nuA} is countably additive.
\end{corollary}
\begin{proof}
Let $E=\bigsqcup_{n=1}^\iii E_n$ for some sequence $(E_n)_{n=1}^\iii$ of mutually disjoint elements in the $\sigma$-algebra $\M$. Now, due to Lemma \ref{adjungovani} we have
\begin{multline}\label{AzvezdaPRBR}
\nu_{\AAA^*}\left(\bigsqcup_{n=1}^\iii E_n\right)=\int_E\AAA^*\,d\mu=\left(\int_E\AAA\, d\mu\right)^*=(\nu_\AAA(E))^*=\left(\sumn\nu_\AAA(E_n)\right)^*\\=\sumn(\nu_\AAA(E_n))^*=\sumn\left(\int_{E_n}\AAA \,d\mu\right)^*=\sumn\int_{E_n}\AAA^*\,d\mu=\sumn\nu_{\AAA^*}(E_n),
\end{multline}
where the countable additivity of $\AAA$ follows from Theorem \ref{ACmuPb}.  The first equality in the second line of    \eqref{AzvezdaPRBR} follows from the fact that conjugation preserves the norm convergence.
\end{proof}

Whenever $X^*$ contains an isomorphic copy of $\mathfrak{c}_0$ there is a family in $\ell^1_s(\BXY)$ which does not generate countably additive o. v. measure. This follows from  the proof of Theorem \ref{l1sbuk}.

Theorem \ref{ACmuPb} is one of the main results of the paper. We will outline another proof under somewhat stronger assumptions, based on different techniques, at the end of this subsection. First, since $\nu_\AAA$ is countably additive and vanishes on sets of measure zero, we can apply Theorem \ref{petismu}. This gives the following corollary.
\begin{corollary}
Under the assumptions of Theorem \ref{ACmuPb}, we have $\nu_\AAA \ll \mu$.
\end{corollary}

The next two theorems give deeper insight into families $\AAA$ in the normed space $L^1_s(\Omega, \mu, \mathcalb B (X,Y))$, under assumption that $X$ contains no isomorphic copy of $\ell^1$.
The first one (Theorem \ref{UCA}) shows that family $\{ \| \AAA_t x \|_Y \, d\mu(t) : x \in B_X \}$ of measures is uniformly countably additive.  
The second one (Theorem \ref{vecmes}), which follows from the first one, states that $E \mapsto \chi_E \AAA$ is an $L^1_s(\Omega, \M, \mu, \BXY)$-valued vector measure, and its Corollary  \ref{UINT} tells us that $\{ \| \AAA x \|_Y : x \in B_X \}$ is a uniformly integrable family. This corollary allows us to give a different proof of Theorem \ref{ACmuPb} for spaces $X$ that contain no isomorphic copy of $\ell^1$.


For every $x\in X$ we have a finite measure $\mu_x\colon\M\rightarrow[0,+\iii)$ defined by \begin{equation}\label{mux}\mu_x(E)=\int_E\|\AAA_tx\|_Y\,d\mu(t)\,\,\,\text{for all}\,\,\, E\in\M.\end{equation}

\begin{theorem}\label{UCA}
Let $\AAA\in L^1_s(\Omega, \mu, \mathcalb B (X,Y))$, where $X$ does not contain a copy of $\ell^1$. Then, the family of measures $\mathcalb{F}=\{\mu_x:\|x\|_X = 1\}$ is uniformly countably additive.
\end{theorem}
\begin{proof}
We argue by contradiction. As $\mathcalb{F}$ is not uniformly countably additive, due to \cite[Proposition~17, Page~8]{DU}, there exist $\varepsilon > 0$, a sequence $(E_n)_{n=1}^\iii$ of disjoint measurable sets and unit vectors $(x_n)_{n=1}^\iii$ in $X$ such that
$$\mu_{x_n}(E_n) = \int_{E_n}\|\AAA_tx_n\|_Y d\mu(t)\geqslant\varepsilon, \qquad n \in \N.$$
The family $\mathcalb{F}$ is uniformly bounded, as $\mu_x(\WWW)\leqslant\|\AAA\|_1,$ for every $\|x\|=1$. Hence, by Rosenthal's Lemma \cite[Lemma~1, Page 18]{DU}, there is a subsequence $(n_k)_{k=1}^\iii$ such that
$$\int_{\bigsqcup_{j\neq k} E_{n_j}}\|\AAA_tx_{n_k}\|_Y d\mu(t) = \mu_{x_{n_k}} \left(\bigsqcup_{j\neq k} E_{n_j}\right)<\frac{\varepsilon}{2}, \qquad k \in \N.$$
Let $\gamma\in\ell^1$ be arbitrary. The same calculation as in the proof of \cite[Theorem 4, Page 104]{DU}  (where $f_k(t)=\AAA_tx_{n_k}$, for $k\in\N$) provides us with an estimate
    \begin{equation}\label{epspola}
        \int_\WWW\left\|\sum_{k=1}^\iii \gamma_k\AAA_t x_{n_k} \right\|_Yd\mu(t) \geqslant\frac{\varepsilon}{2} \sum_{k=1}^\iii |\gamma_k|.
    \end{equation}
    On the other hand, as $\AAt$ is bounded and $\sum_{n=1}^\iii \gamma_k x_{n_k}$ (absolutely) converges in $X$, we also get $\sum_{k=1}^\iii \gamma_k \AAt x_{n_k} = \AAt (\sum_{k=1}^\iii \gamma_k x_{n_k})$ for every $t\in\WWW$. Next, 
    $$\int_\WWW \left\|\AAA_t \left(\sum_{k=1}^\iii \gamma_k x_{n_k}\right)\right\|_Y d\mu(t) \leqslant \|\AAA\|_1\cdot \left\|\sum_{k=1}^\iii \gamma_k x_{n_k}\right\|_X,$$
 and \eqref{epspola} imply that
 \begin{equation}\label{below}
     \left\|\sum_{k=1}^\iii \gamma_k x_{n_k}\right\|_X \geqslant\frac{\varepsilon}{2\|\AAA\|_1}\cdot\sum_{k=1}^\iii |\gamma_k|.
 \end{equation}
Furthermore, we have
\begin{equation}\label{above}
    \left\|\sum_{k=1}^\iii \gamma_k x_{n_k}\right\|_X\leqslant\sum_{k=1}^\iii \|\gamma_k x_{n_k}\|_X=\sum_{k=1}^\iii |\gamma_k|,
\end{equation}
   Estimates \eqref{below} and \eqref{above} produce a basic sequence $(x_{n_k})_{k=1}^\infty$ in $X$ which spans an isomorphic copy of $\ell^1$, which is a contradiction, as $X$ does not contain $\ell^1$.
\end{proof}

Theorem \ref{vecmes} below provides a curious phenomenon: a countably additive vector measure with values in a normed space which is not necessarily complete. 

\begin{theorem}\label{vecmes}
    Let $\AAA \in L^1_s(\Omega, \mu, \mathcalb B (X,Y))$, where $X$ does not contain a copy of $\ell^1$. 
    Then the mapping $\lambda_\AAA : \M \to L^1_s(\Omega, \mu, \mathcalb B (X,Y))$ defined by $\lambda_\AAA (E) = \chi_E \AAA$ is a countably additive $\mu$-continuous vector measure with values in $L^1_s(\Omega, \mu, \mathcalb B (X,Y))$.
\end{theorem}
\begin{proof}
    Clearly $\lambda_\AAA$ is finitely additive. Assume $E$ is a disjoint union of sets $E_n$ in $\M$. From the definition of norm in $L^1_s(\Omega, \mu, \mathcalb B (X,Y))$, Theorem \ref{UCA} and \eqref{UCAdef} we have the following equality:
    \begin{equation*}
        \begin{split}
            \lim_{n\to\infty} \left\| \lambda_\AAA (E) \!-\! \sum_{k=1}^n \lambda_\AAA (E_k) \right\|_{1}\!\!& = \lim_{n\to\infty} \sup_{\| x \|_X = 1} \int_{\bigsqcup\limits_{k=n}^\iii E_k} \!\| \AAA_t x \|_Y \, d\mu (t) = 0.
        \end{split}  
    \end{equation*}

Moreover, if $\lambda_\AAA$ is not $\mu$-continuous, then there exist $\varepsilon > 0$ and a sequence of sets $E_n$ in $\M$ such that $\mu(E_n) \leqslant\frac{1}{2^n}$ and $\|\lambda(E_n)\|_1 \geqslant \varepsilon$. 
Set $F_n = \bigcup_{k=n}^\iii E_k$, then we have $\mu(F_n) \leqslant \frac{1}{2^{n-1}}$ and $\|\lambda_\AAA(F_n)\|_1\geqslant \|\lambda_\AAA(E_n)\|_1\geqslant\varepsilon$. But, Theorem \ref{UCA} and \eqref{UCAekv} imply $\lim_{n\to\infty}\|\lambda_\AAA(F_n)\|_1 = 0$ and we arrived at a contradiction.
 \end{proof}

\begin{corollary}\label{UINT}
Let $\AAA \in L^1_s(\Omega, \mu, \mathcalb B (X,Y))$, where $X$ does not contain a copy of $\ell^1$. Then the family of scalar functions $\{\|\AAA x\|_Y:\|x\|_X\leqslant R\}$ is uniformly integrable, for every $R>0$.
\end{corollary}
\begin{proof}
By Theorem \ref{vecmes}, $\lambda_\AAA \ll \mu$, hence for every $\varepsilon > 0$ there is $\delta > 0$ such that $\mu(E) < \delta$ implies $\|\AAA\chi_E\|_1 <\frac{\varepsilon}{R}$ i.\,e. $\int_E\|\AAt x\|\, d\mu(t) < \frac{\varepsilon}{R}$ for all vectors $x$ in $B_X$. Therefore, if $\mu(E)<\delta$ we have
$$\int_E \|\AAt x\|_Y \,d\mu(t) \leqslant R \int_E \left\|\AAt\left(\frac{x}{\|x\|_X}\right)\right\|_Y d\mu(t) < \varepsilon,$$
for all $\|x\|_X\leqslant R$, implying proclaimed uniform integrability.
\end{proof}

\begin{proof}[Another proof of Theorem \ref{ACmuPb}.]

Since integration is a linear and bounded operator from $L^1_s(\Omega, \mu, \BXY)$ to $\BXY$ (see \eqref{normEst}), Theorem \ref{vecmes} implies Theorem \ref{ACmuPb} for spaces $X$ that contain no copy of $\ell^1$.
\end{proof}

We conclude this subsection with another application of Theorem \ref{UCA} that provides an alternative to Lemma \ref{sigfin}. 
For a family $\AAA$ in $L^1_s(\Omega, \mu, \mathcalb B (X,Y))$ the set $\{ t \in \Omega : \AAt\neq 0\}$ is not necessarily $\sigma$-finite even for Hilbert spaces, see Example \ref{longline}. Nevertheless, in the case where $X$ does not contain copy of $\ell^1$, it has to be $\sigma$-finite with respect to $L_s^1$ norm, as the following corollary shows.

\begin{corollary}\label{sigfinl1}
Let $\AAA \in L^1_s(\Omega, \mu, \mathcalb B (X,Y))$, where $X$ does not contain a copy of the space $\ell^1$. Then, there are disjoint sets $S,N\in\M$ such that $\WWW = S\sqcup N$, where the set $S$ is $\sigma$-finite and $\|\AAA \chi_N\|_1 = 0$.
\end{corollary}

\begin{proof}
Since the measures $\mu_x$ from \eqref{mux}
are uniformly countably additive for $\|x\|_X=1$, the same arguments as in the proof of \cite[Theorem 1.2.4, Page 12]{DU} shows that for every $\varepsilon > 0$ there exists unit vectors $x_1,\ldots,x_n\in X$ such that $\mu_{x_k}(E) = 0$ for every $k\in\{1,\ldots,n\}$ implies $\mu_x(E)<\varepsilon$ for every unit vector $x\in X$.
Therefore, as in the proof of the mentioned theorem, for $\varepsilon = \frac{1}{n}$ there exists $x_1^n,\ldots, x^n_{k_n}\in X$ such that $\mu_{x^n_i}(E) = 0$ for every $i\in\{1,\ldots,k_n\}$ implies $\mu_x(E)<\frac{1}{n}$ for every unit vector $x\in X$. Hence, if $E\in\M$ satisfies $\mu_{x_i^n} (E) = 0$ for every $n\in\mathbb{N}$ and every $i\in\{1,\ldots,k_n\}$ then $\mu_x(E)< \frac{1}{n}$ for every $n\in\mathbb{N}$ and every $x\in B_X$, implying that $\mu_x (E) = 0$. We set
$$S_i^n = \{t\in\WWW:\|\AAt x_i^n\|_Y\neq 0\},\qquad n\in\mathbb{N},\,\, i\in\{1,\ldots,k_n\}.$$
As the scalar function $\Omega\ni t\to\|\AAt x\|_Y\in\R$ is $\mu$-integrable for every vector $x$ in $X$, the set $S_i^n$ is measurable and $\sigma$-finite for every $n\in\mathbb{N}$ and every $i$ in the set $\{1,\ldots,k_n\}$. Hence, the set
$$S = \bigcup_{n=1}^\infty\bigcup_{i=1}^{k_n}
S_i^n$$
is measurable and $\sigma$-finite since the union is countable. Finally, letting $N=\WWW\setminus S$, by the construction of the set $S$ we obtain
$$\int_N \left\|\AAt x_i^n\right\|_Y\, d\mu(t) = \mu_{x^n_i}(N) = 0,\qquad n\in\N,\,\,i\in\{1,\ldots,k_n\}.$$
Therefore, we have the equality $$\mu_{x}(N) = \int_N \|\AAt x\|_Y\,d\mu(t) = 0$$ for every unit vector $x\in X$ implying that $\|\AAA\chi_N\|_1 = 0$.
\end{proof}
If we do not assume that $X$ does not contain $\ell^1$ and also do not assume that $X$ is separable, then the conclusion of the above Corollary does not hold. In fact, take $\Omega = \mathbb R$, $\M = \mathcal P(\mathbb R)$ with counting measure $\mu$ and set $X = \ell^1(\mathbb R)$. 
Define $\AAA = (\AAA_t)_{t \in \mathbb R}$ by $\AAA_t x = x_t e_t$. Then $\AAA \in L^1_s(\mathbb R, \mathcal P(\mathbb R), \mu, X)$ and in fact $\| \AAA \chi_E \|_1 = 1$ for every nonempty $E \subset \mathbb R$. Therefore, there is no countable subset $S \subset \mathbb R$ such that $\| \AAA \chi_{S^c} \|_1 = 0$.

\subsection{Approximation results}
In this subsection we consider approximation problems for elements of $L_s^1(\Omega,\mu,\BXY)$ by families of bounded elements from $L_s^1(\Omega,\mu,\BXY)$, as well as by sequences of simple families from this space. We provide sufficient conditions under which this is possible, and we also present an example showing that approximation in the norm of $L_s^1(\Omega,\mu,\BXY)$ by simple families cannot always be achieved.

Theorem \ref{AproxP=1} below is an approximation result for the space $L_s^1(\Omega,\mu,\BXY)$ by bounded families. 
A similar theorem was recently obtained for functions with values in spaces of measures (see \cite[Theorem 4.1]{ABK2}).

\begin{theorem}\label{AproxP=1} Assume $X$ is separable and contains no isomorphic copy of $\ell^1$. Then, bounded families in $L_s^1(\Omega,\mu,\BXY)$ are dense in $L_s^1(\Omega,\mu,\BXY)$. 
\end{theorem}
\begin{proof}
We fix $\AAA$ in $L_s^1(\Omega,\mu,\BXY)$. Since $X$ separable, $\Omega \ni t \mapsto \| \AAA_t \| \in \R$ is a $\mu$-measurable function. Hence we have an increasing sequence of measurable sets $\Omega_n =
\{t\in\Omega:\|\AAA_t\|\leqslant n\} $ whose union is $\Omega$.
We define bounded functions $\AAA_n=\chi_{\Omega_n}\cdot\AAA$ in $L_s^1(\Omega,\mu,\BXY)$ for all $n\in\N$. By Theorem \ref{vecmes} we have

\begin{equation*}
  \lim_{n\to\infty}\|\AAA-\AAA_n\|_1 = \lim_{n\to\infty} \| \lambda_\AAA (\Omega) - \lambda_\AAA(\Omega_n)\|_1 = 0. \qedhere  
\end{equation*}
\end{proof}

Every $\AAA$ in $L_s^1(\Omega,\mu,\BXY)$ has the property that for every $x \in X$ the family $\AAA x$ in $Y$ can be approximated by a sequence of simple $Y$-valued functions, which follows directly from Bochner integrability of the family $\AAA x$ (see Definition \ref{OsnovnaDef}). 
Theorem \ref{SchAppr} provides sufficient conditions under which the approximating simple family of operators can be chosen independently of the vector $x$ in $X$. For this we need the following notion.

It is said that a sequence $(X_n)_{n=1}^\infty$ of finite dimensional subspaces of a Banach space $X$ form \textit{F.\,D.\,D.} (\textit{finite dimensional Schauder decomposition}) of $X$ if for every $x$ in $X$ there is a unique sequence $(x_n)_{n=1}^\infty$ such that $x_n \in X_n$ for all $n\in\N$ and $x = \sum_{n=1}^\infty x_n$. 
In that case, we have a sequence of projections $(P_n)_{n=1}^\infty$ which map $x = \sum_{n=1}^\infty x_n$ to $\sum_{k=1}^n x_k$ and one shows that $M = \sup_{n\in\N} \| P_n \| < + \infty$, the range of $P_n$ is the direct sum $X_n'=\bigoplus_{k=1}^nX_k$. 
We have natural inclusions $j_n \colon X_n' \rightarrow X$ defined by $j_n(x)=x$ for $x\in X_n'$, as well as natural projections $\pi_n\colon X\rightarrow X_n'$  defined by $\pi_n(x)=P_n(x)$ for $x\in X$. Then we have $\pi_n\circ j_n=id_{X_n'}$, $j_n\circ\pi_n=P_n$ and $\| \pi_n \| = \| P_n \| \leqslant M$ for every $n\in\N$. Whenever we have an F.\,D.\,D. $(X_n)_{n=1}^\infty$ for $X$ we use the above notation.

The decomposition $(X_n)_{n=1}^\infty$ is called \textit{shrinking} if $\lim_{n\to\infty} \|P_n^* x^* - x^*\|_{X^*} = 0$ for every $x^* \in X^*$. For more details, see \cite[Section~1.g]{LTz}.

\begin{proposition}\label{MerljivAprox}
      Assume $(X_n)_{n=1}^\infty$ is a shrinking F.\,D.\,D. for $X$ and $\AAA\in\mathcal{M}_s(\WWW,\M,\mu,\BXY)$ such that $\AAA_t\in\KXY$ for $\mu$-almost all $t\in\Omega$. Then, $\AAA\in\mathcal{M}(\WWW,\M,\mu,\KXY)$. 
\end{proposition}
\begin{proof}
    Set $\Btn=\AAA_tP_nj_n\in\mathcalb{B}(X_n',Y)$ for all $t\in\WWW$ and $n\in\N$. Then, $\BBB^{(n)}\in\mathcal{M}_s(\WWW,\M,\mu,\mathcalb{B}(X_n',Y))$ and thus $\BBB^{(n)}\in\mathcal{M}(\WWW,\M,\mu,\mathcalb{B}(X_n',Y))$, due to Proposition \ref{BochnerFinite}, as $\dim X_n'<+\iii$. It follows that $\BBB^{(n)}\pi_n\in\mathcal{M}(\WWW,\M,\mu,\mathcalb{B}(X,Y))$, due to Lemma \ref{Complemented}. For every $t\in\WWW$ we have
    $$\Btn\pi_n=\AAA_tP_nj_n\pi_n=\AAA_tP_n^2=\AAA_tP_n\rightarrow\AAA_tI_X=\AAA_t,$$
    in strong operator topology. Next, as $\AAA_t\in\KXY$ (and thus $\AAA^*_t\in\mathcalb{K}(Y^*,X^*)$) for $\mu$-a.\,e. $t\in\WWW$ and $P_n^*\rightarrow I_{X^*}$ strongly (as $(X_n)_{n=1}^\iii$ is shrinking), it follows that $$\left\|\Btn\pi_n-\AAA_t\right\|=\|\AAA_tP_n-\AAA_t\|=\|(P^*_n-I_{X^*})\AAA_t^*\|\rightarrow0,$$
    for $\mu$-a.\,e. $t\in\WWW$ (see arguments\footnote{We warn the reader that there are misplaced parenthesis.} before Theorem 3.14 in \cite{MAMK}). Thus, $\AAA$ is $\mu$-a.\,e. limit of strongly $\mu$-measuarable functions. Final conclusion follows from \cite[Corollary~1.1.23]{ABS}.
\end{proof}

\begin{theorem}\label{SchAppr}
    Let $(X_n)_{n=1}^\infty$ be an F.\,D.\,D. for $X$ and assume also that $X$ does not contain an isomorphic copy of $\ell^1$. Then for every  
   $\AAA\in L^1_s(\Omega, \mu, \mathcalb B(X, Y))$
there is a sequence $\left(\mathcal S^{(n)}\right)_{n=1}^\infty$ of $\mu$-simple functions in $L^1_s(\Omega, \mu, \mathcalb B(X, Y))$ such that
   \begin{equation}\label{appsimple}
\displaystyle\lim_{n\to\infty} \int_\Omega \left\| \AAA_tx - \mathcal{S}_t^{(n)} x \right\|_Y d\mu(t) = 0\,\,\,\text{for all}\,\,\, x\in X.
\end{equation}
If, in addition, F.\,D.\,D. $(X_n)_{n=1}^\iii$ is shrinking, and $\AAA_t \in \KXY$ for every $t \in \Omega$,  then $\AAA$ is in the $L^1_s$-closure of the set of $\mu$-simple $\KXY$-valued functions, i.\,e. the limit in \eqref{appsimple} is uniform over $x \in B_X$. 
\end{theorem}
\begin{proof}
Set $\BBB_t=\AAA_tP_nj_n\in\mathcalb{B}(X_n',Y)$ for all $t\in\WWW$ and $n\in\N$. 
Next, we have $$\left\|\Btn x'\right\|_Y = \|\AAA_t P_nj_nx'\|_Y = \|\AAA_tx'\|_Y$$ for every $x'\in X_n'$, implying $\BBB \in L^1_s(\WWW,\mu,\mathcalb{B}(X_n',Y))$ for all $n\in\N$. Due to Proposition \ref{BochnerFinite}, there exist  $\mu$-simple $\mathcalb{B}(X_n',Y)$-valued functions $\CCC^{(n)}$ such that 
$$\int_\WWW\left\|\Btn-\CCC^{(n)}_t\right\|d\mu(t)<\frac{1}{n}, \qquad n \in \N.$$
Finally, let $\Stn=\CCC^{(n)}_t\pi_n\in\BXY$ for all $t\in\WWW$. Then, for all $n\in\N$ we get
\begin{equation}\label{approx}
\begin{split}
\int_\WWW&\left\|\AAA_tx\!-\!\Stn \!x\right\|_Yd\mu(t)\!=\!\!\!\int_\WWW\left\|\!\left(\!\AAA_tx\!-\!\Btn\pi_nx\!\right)\!\!+\!\!\left(\!\Btn\pi_nx\!-\!\CCC^{(n)}_t\pi_nx\!\right)\!\right\|_Y d\mu(t)\\
&\leqslant\int_\WWW\left\|\AAA_tx-\Btn\pi_nx\right\|_Y d\mu(t) + \int_\WWW\left\|\Btn\pi_nx-\CCC^{(n)}_t\pi_nx\right\|_Y\,d\mu(t)\\
&\leqslant\int_\WWW\|\AAA_tx-\AAA_tP_nj_n\pi_nx\|_Y\,d\mu(t)+\int_\WWW\left\|\Btn-\CCC^{(n)}_t\right\|\|\pi_n\|\|x\|_X\,d\mu(t)\\
&\leqslant\int_\WWW\left\|\AAA_tx-\AAA_tP_n^2x\right\|_Y \,d\mu(t) + M \cdot \|x\|_X \cdot \int_\WWW\left\|\Btn-\CCC^{(n)}_t\right\|d\mu(t)\\
&<\int_\WWW\|\AAA_tx-\AAA_tP_nx\|_Y\,d\mu(t)+\frac{M\cdot\|x\|_X}{n}, \qquad x \in X.
\end{split}
\end{equation} 
For the proof of the first part of the theorem, it suffices to show that 
\begin{equation}\label{fpthm}
\lim_{n\to\infty}\int_\WWW\|\AAA_t(I_X-P_n)x\|_Y\,d\mu(t) = 0, \qquad \forall x \in X.
\end{equation}
Let us fix $\varepsilon > 0$. By Theorem \ref{vecmes}, there is $\delta > 0$ such that $\| \chi_E \AAA \|_1 < \varepsilon$ whenever $\mu(E) < \delta$, also implying that (as $\|P_n x\|_X\leqslant M\|x\|_X$)
\begin{equation}\label{chiEAAA}
\begin{split}
\|\chi_E \AAA P_n \|_1 & = \sup_{\|x\|_X\leqslant1} \int_E \|\AAA_t P_nx\|_Y\,d\mu(t) \leqslant \sup_{\|x\|_X\leqslant M} \int_E \|\AAA_t x\|\,d\mu(t)\\
&  = M\cdot \|\chi_E\AAA\|_1 < M\cdot\varepsilon, \qquad \mu(E) < \delta.   
\end{split}
\end{equation}
Since $P_n \rightarrow I_X$ strongly, the sequence of scalar functions $\varphi_n(t)=\|\AAA_t(I_X-P_n)x\|_Y$ converges to $0$ pointwise. 
As $\|(I_X-P_n)x\|_X\leqslant(M+1)\cdot\|x\|_X$ for all $n\in\N$, we have
\begin{equation*}
\{\varphi_n:n\in\N\}\subseteq\{\|\AAA y\|_Y:\|y\|_X\leqslant(M+1)\cdot\|x\|_X\},
\end{equation*}
and Corollary \ref{UINT} implies uniform integrability of $\{\varphi_n:n\in\N\}$. As $X$ has an F.\,D.\,D, it is separable, hence by Lemma \ref{sigfin} we can assume that $\mu$ is $\sigma$-finite.  Thus, there are $\WWW_k\in\M$ such that $\mu(\WWW_k)<+\iii$ and $\WWW = \bigsqcup_k \WWW_k$. 
By Theorem \ref{vecmes}, there exists $K \in \mathbb{N}$ such that $\|\chi_{\bigsqcup_{k>K}\WWW_k}\AAA\|_1 < \varepsilon$ and by \eqref{chiEAAA}, we also have $\|\chi_{\bigsqcup_{k>K}\Omega_k}\AAA P_n\|_1 < M\cdot\varepsilon$ for every $n\in\mathbb{N}$. Combining previous estimates we get
\begin{multline*}
\int_\WWW \| \AAA_t x - \AAA_tP_nx \|_Y\,d\mu(t) = \int_{\bigsqcup\limits_{k\leqslant K} \!\Omega_k} \|\AAA x -\AAA P_nx\|_Y\,d\mu + \\
+ \int_{\bigsqcup\limits_{k>K}\!\Omega_k} \|\AAA x - \AAA P_nx\|_Y\,d\mu
\leqslant \int_{\bigsqcup\limits_{k\leqslant K}\! \Omega_k} \|\AAA x-\AAA P_nx\|_Y\,d\mu + \varepsilon \|x\|_X + M \varepsilon \|x\|_X,
\end{multline*}
for every $x \in X$. As $\bigsqcup_{k\leqslant K} \Omega_k$ is a set of finite measure, uniform integrability of $\{\varphi_n :n\in\N\}$ allows us to apply Vitali Convergence Theorem, yielding
$$\lim_{n\to\infty}\int_{\bigsqcup_{k\leqslant K} \Omega_k}\|\AAA_tx-\AAA_tP_nx\|_Y\,d\mu(t)= 0$$
and finally by taking $\varlimsup_{n\to\iii}$ in the above estimates we get
$$\varlimsup_{n\to\iii} \int_\WWW\|\AAA_tx-\AAA_tP_nx\|_Y\,d\mu(t) \leqslant (M+1)\varepsilon\|x\|_X.$$
As $\varepsilon>0$ was arbitrary, we obtain \eqref{fpthm}. 

Next we discuss the case of compact operators. Since $(X_n)_{n=1}^\infty$ is shrinking, $X^*$ is separable and therefore $X$ does not contain a copy of $\ell^1$. By taking the supremum over $x \in B_X$ in \eqref{approx} we obtain
\begin{equation}\label{AtStn}
\begin{split}
\left\| \AAA - \mathcal S^{(n)} \right\|_{1}& \leqslant \sup_{\|x\|_X\leqslant1} \int_\Omega \|\AAA_tx-\AAA_tP_nx\|_Y \,d\mu(t) + \frac{M}{n}.
\end{split}
\end{equation}
We choose $\varepsilon > 0$. By Theorem \ref{vecmes}, there is a $\delta > 0$ such that $\| \chi_E \AAA\|_1 < \varepsilon$ whenever $\mu(E) < \delta$. As in the noncompact case, there is a set $F\in\M$ of finite measure such that $\| \AAA \chi_{\WWW\setminus F} \|_1 < \varepsilon$. Moreover, by \eqref{chiEAAA}, the estimate $\|\chi_{\WWW\setminus F} \AAA P_n\|_1 < M\cdot\varepsilon$, holds for every $n\in\mathbb{N}$. If we apply previous estimates to \eqref{AtStn} we get 
\begin{align}\label{AtStn1}
&\!\!\!\left\| \AAA - \mathcal S^{(n)} \right\|_1 \leqslant \sup_{\|x\|_X\leqslant1} \int_F \| \AAA_t x - \AAA_t P_n x \|_Y \,d\mu(t) \notag\\
&+ \sup_{\| x \|_X \leqslant 1} \int_{\WWW\setminus F}\| \AAA_t x- \AAA_t P_n x\|_Y \,d\mu(t)\!+\!\frac{M}{n}\notag\\
&\leqslant \sup_{\|x\|_X\leqslant1}\int_F \| \AAA_t x - \AAA_t P_nx \|_Y \,d\mu(t)  + \|\AAA P_n\chi_{\WWW\setminus F}\|_1+ \|\AAA\chi_{\WWW\setminus F}\|_1 + \frac{M}{n}\notag\\
&\leqslant \sup_{\|x\|_X\leqslant1}\int_F \| \AAA_t x - \AAA_t P_nx \|_Y \,d\mu(t) + (M+1)\cdot\varepsilon + \frac{M}{n}.
\end{align}
Next, for $n\in\N$ define $\mu$-measurable functions $\psi_n\colon\WWW\rightarrow[0,+\iii)$ by $$\psi_n(t)=\| \AAA_t(I_X -  P_n)\|,\qquad t\in\Omega.$$These function are $\mu$-measurable since $X$ is separable.
As F.\,D.\,D. $(X_n)_{n=1}^\iii$ is shrinking, we have that $P_n^*\rightarrow I_{X^*}$ strongly. Thus, we have $$\lim_{n\to\infty}\psi_n(t)=\lim_{n\to\infty}\|(I_{X^*}-P_n^*)\AAA_t^*\|=0,$$because $\AAA_t\in\mathcalb{K}(Y^*,X^*)$ for all $t\in\WWW$.
Thus, using Egorov's theorem on a set $F$, there is a set $F_\delta \in \M$ such that $\mu(F_\delta) < \delta$ and $\psi_n \rightrightarrows 0$ on $F \setminus F_\delta$. Finally, if we apply previous estimates to \eqref{AtStn1}, for all $n\in\N$, we get
\begin{equation*}
\begin{split}
&\left\| \AAA - \mathcal S^{(n)} \right\|_1 \leqslant \sup_{\|x\|_X\leqslant1}\int_{F\setminus F_\delta} \| \AAA_t x - \AAA_t P_nx \|_Y \,d\mu(t)\! \\  + &\sup_{\|x\|_X\leqslant1}\int_{F_\delta} \| \AAA_t x - \AAA_t P_nx \|_Y \,d\mu(t)+ (M+1)\cdot\varepsilon+\!\frac{M}{n}\\
\leqslant &\sup_{\|x\|_X\leqslant1}\int_{F\setminus F_\delta}\!\!\| \AAA_t x - \AAA_t P_nx \|_Y \,\!d\mu(t)\!+\! \|\AAA\chi_{F_\delta}\|_1 \!+ \!\|\AAA P_n\chi_{F_\delta}\|_1+ (M+1)\cdot\varepsilon + \frac{M}{n}\\
\leqslant &\,\mu(F) \cdot \sup_{t\in F\setminus F_\delta}\psi_n(t) + 2(M+1)\cdot\varepsilon+\frac{M}{n},
\end{split}
\end{equation*}
where the last inequality follows from the \eqref{chiEAAA} as $\mu(F_\delta) < \delta$. Next, similarly to noncompact case, we get
\begin{equation*}
    \begin{split}
       \varlimsup_{n \to \infty} \left\| \AAA - \mathcal S^{(n)} \right\|_1 
       \leqslant2(1+M)\cdot\varepsilon, 
    \end{split}
\end{equation*}
completing the proof.
\end{proof}
\begin{remark}
    It is still possible to obtain a result if we drop assumption that $X$ does not contain a copy of $\ell^1$, but in that case we have to impose a fairly strong condition: $\int_\Omega \| \AAA_t \| \, d\mu(t) < +\infty$. First, note that $\| \AAA_t \|$ is a measurable function due to separability of $X$. 
    Then, we can use Dominated Convergence Theorem in \eqref{fpthm} with integrable majorant $g(t) = (1+M)\cdot\| \AAA_t \|$.

    In the case of compact operators, condition $\int_\Omega \| \AAA_t \| \, d\mu(t) < +\infty$ implies, in view of Proposition \ref{MerljivAprox}, that $\AAA$ is in $L^1(\Omega, \M, \mu, \KXY)$.
\end{remark}

Note that Theorem \ref{SchAppr} does not give approximation in $L^1_s$ norm, in general the approximation is not uniform over $x \in B_X$. Example \ref{Rademacher} shows that one can not get approximation in $L^1_s$-norm by simple functions even if we restrict to Hilbert spaces (which certainly do not contain $\ell^1$) and additionally assume that $\int_\Omega \| \AAA_t \| \, d\mu(t) < +\infty$ and $\mu(\Omega) < + \infty$. Of course, the operators appearing in Example \ref{Rademacher} are not compact.

\begin{example}\label{Rademacher}
We will define $\AAA \in L^1_s([0,1],\mathcalb{B}(\ell^2))$ which is not in the $L^1_s$-norm closure of the set $S$ of simple functions. 
Let $r_n : [0,1]\to \{-1,1\}$ be the $n$-th Rademacher function, each $r_n$ is $m$-measurable, takes alternatively values $\pm 1$ on dyadic intervals of length $2^{-n}$.
Furhter, let $(e_n)_{n=1}^\infty$ be the standard orthonormal basis of $\ell^2$.
Define, for every $0 \leqslant t \leqslant 1$, a self-adjoint operator $\AAA_t\in \mathcalb{B}(\ell^2)$ by
$$\AAA_t e_n = r_n(t)e_n, \qquad n\in\N. $$
Equivalently, for $x = (x_n)_{n=1}^\infty \in \ell^2$, we have $\AAA_t x = (r_n(t)x_n)_{n=1}^\infty$.
Thus $\AAA_t$ is a diagonal operator with diagonal entries $r_n(t)\in\{-1,1\}$, hence it is unitary and self-adjoint for all $0 \leqslant t \leqslant 1$.
It is easily seen that $\AAA = (\AAA_t)_{0 \leqslant t \leqslant 1}$ is pointwise $m$-measurable. Since $\int_0^1 \| \AAA_t x \|_{\ell^2} \, dm(t) = \| x \|_{\ell^2}$ it is immediate that $\AAA$ belongs to $L^1_s([0,1],\mathcalb{B}(\ell^2))$ and in fact $\| \AAA \|_1 = 1$. We prove that
\begin{equation}\label{dist}
    \textup{dist}(\AAA, S) = 1.
\end{equation}
Since $\| \AAA \|_1 = 1$ clearly $\textup{dist}(\AAA, S) \leqslant 1$. Let $\BBB_t=\sum_{k=1}^m \chi_{A_k}(t)\,T_k$
be an arbitrary $m$-simple o.\,v. function. This means the sets $A_1,\dots,A_m$ are Lebesgue measurable of positive measure, pairwise disjoint, satisfy
$[0,1]=\bigsqcup_{k=1}^m A_k$ and $T_1, \ldots, T_m$ are operators in $\mathcalb B(\ell^2)$. Let us fix $\varepsilon > 0$.
For $k\in\{1,\dots,m\}$ and $n\in  \N$ we have
$$\|(\AAA_t - \BBB_t )e_n\|_{\ell^2} = \|r_n(t)e_n - T_ke_n\|_{\ell^2}, \qquad t \in A_k.$$
For every $n \in \N$ we can split \(A_k\) into two measurable parts
$$ A_{k,n}^+ := A_k\cap \{t\in[0,1]: r_n(t)=1\},
\,\,\,
A_{k,n}^- := A_k\cap \{t\in[0,1]: r_n(t)=-1\}. $$

Since for all $n,k\in\N$ we have the equalities
$$\int_{A_k} r_n(t)\,dm(t) = m(A_{k,n}^+) - m(A_{k,n}^-), \,\,\, m(A_{k,n}^+) + m(A_{k,n}^-)=m(A_k),$$
and $\lim\limits_{n \to \infty} \int_{A_k} r_n(t)\,dm(t) = 0$ for every $k\in\{1,\ldots,m\}$, it follows that
$$
\lim_{n\to\infty}m(A_{k,n}^+) = \frac{m(A_k)}{2}, \qquad  \lim_{n\to\infty}m(A_{k,n}^-) = \frac{m(A_k)}{2}, \qquad k\in\{1,\ldots,m\}.
$$
Hence, there is $N \in \N$ such that
$$ m(A_{k,N}^+), m(A_{k,N}^-)>\frac{m(A_k)}{2}-\frac{\varepsilon}{2m},
\qquad k\in\{1,\ldots,m\},
$$
For such an \(N\), we estimate:
\begin{multline*}
\int_{A_k}\|(\AAA_t-\BBB_t)e_N\|_{\ell^2}\,dm(t) = \|e_N-T_ke_N\|_{\ell^2}\cdot m(A_{k,N}^+) + \|-e_N-T_ke_N\|_{\ell^2} \cdot m(A_{k,N}^-) \\
> \bigl( \|e_N - T_ke_N\|_{\ell^2} + \|-e_N-T_ke_N\|_{\ell^2} \bigr) \cdot \left( \frac{m(A_k)}{2}-\frac{\varepsilon}{2m}\right)\geqslant m(A_k)-\frac{\varepsilon}{m},
\end{multline*}
where we used the fact that $$\|e_n-T_ke_n\|_{\ell^2}+\|-e_n-T_ke_n\|_{\ell^2}\geqslant \|e_n-(-e_n)\|_{\ell^2} = 2.$$
Summing over \(k=1,\dots,m\), we obtain
\begin{align*}
\int_{[0,1]} \|(\AAA_t-\BBB_t)e_N\|_{\ell^2}\,dm(t)
& = \sum_{k=1}^m \int_{A_k}\|(\AAA_t-\BBB_t)e_N\|_{\ell^2}\,dm(t) \\
& \geqslant \sum_{k=1}^m \left(m(A_k)-\frac{\varepsilon}{m}\right) = \sum_{k=1}^m m(A_k) - \varepsilon =
1-\varepsilon.
\end{align*}
Since $\|e_N\|_{\ell^2} = 1$, it follows that
$$ \|\AAA-\BBB\|_1
\geqslant \int_{[0,1]} \|(\AAA_t-\BBB_t)e_N\|_{\ell^2}\,dm(t) \geqslant 1 - \varepsilon.$$
Since \(\varepsilon>0\) was arbitrary, we conclude that
$\|\AAA-\BBB\|_1\geqslant 1$.
\hfill{$\triangle$}
\end{example}
We point out another property of the family $\AAA$ from the above example. 
Namely, the set $\{ \int_E \AAA \, d\mu : E \subset [0,1] \; \mbox{is measurable}\}$ is not relatively compact subset of $\mathcalb B(\ell^2)$. To see this, consider $T_n = \int_{E_n} \AAA \, dm$ where 
$$E_n = \bigsqcup_{j=0}^{2^{n-1}-1} \left[ \frac{2j}{2^n}, \frac{2j+1}{2^n} \right], \qquad n \in \N.$$
Then one checks, using definition of Rademacher's functions, that $T_n = 2^{-1} e_n \otimes e_n^*$ and therefore $\| T_n - T_m \| = 1/2$ whenever $n \not= m$ and this suffices.

In view of Proposition \ref{ClSub}, this shows that $\AAA$ can not be approximated by simple functions, however without delivering sharp estimate \eqref{dist}.

\begin{remark}
The point in the above example is that the operators \(\AAA_t\) oscillate independently in infinitely many coordinates, which is provided by Rademacher functions. A simple function involves only finitely many operator values, so it cannot capture this infinite oscillatory behavior in the \(L^1_s\)-norm.
\end{remark}

\subsection{Compactness results}

Compactness properties of Gelfand integrals were studied for the first time in \cite{MMS}, where operators considered are assumed to belong to certain ideals of compact operators, see also \cite{Matija}. Banach space case was considered in a recent paper \cite{MAMK}, see also \cite{MS} where integration of functions whose values are $(q,p)$-summing operators was investigated. Here we deal with Banach space case and use tensor products and duality as tools.

\begin{theorem}\label{Cpt}
Let $X$ be a Banach space that does not contain a copy of $\ell^1$ and let $\AAA\in L^1_s(\WWW,\mu,\mathcalb{B}(X,Y))$. Assume $\AAA_t\in\KXY$ for a.\,e. $t\in\WWW$. Then, $\int_E\AAA\, d\mu\in\KXY$ for every $E\in\M$.
\end{theorem}
\begin{proof}
By Corollary \ref{sigfinl1} we can assume that $(\WWW,\mu)$ is $\sigma$-finite measure space. We begin with the case $\mu(\WWW)<+\iii$. As $X$ does not contain a copy of $\ell^1$, by Proposition \ref{CCjeK} it suffices to show that $\int_E\AAA d\mu$ is completely continuous. Thus, let $(x_n)_{n=1}^\iii$ be a sequence in $X$ that weakly converges to $0$. As $\AAt$ is completely continuous we have $\varphi_n(t):=\|\AAA_tx_n\|_Y\rightarrow 0$ as $n\to\iii$ for a.\,e. $t\in\WWW$. Moreover, since $\sup_{n\in\N}\|x_n\|_X<+\iii$, the family $\{\varphi_n:n\in\N\}$ is uniformly integrable due to Corollary \ref{UINT} yielding
$$\lim_{n\to\infty}\left\|\left(\int_E\AAA\, d\mu\right)x_n\right\|_Y\leqslant\lim_{n\to\infty}\int_E\|\AAA_tx_n\|_Y\,d\mu(t)=0,$$
by the Vitali's Convergence Theorem, proving that $\int_E\AAA d\mu\in\KXY$.

Next, let $\mu(E)=+\iii$. As $X^*$ does not contain a copy of $\mathfrak{c}_0$ (see, for example \cite[Theorem 10, Page 48]{D}), o. v. measure $\nu_\AAA$ is countably additive by Theorem \ref{ACmuPb}. Since $\mu$ is $\sigma$-finite there is a decomposition $E = \bigsqcup_{n=1}^\infty E_n$ where $\mu(E_n) < +\iii$ for all $n\in\N$. Finally, as $\nu_\AAA (E) = \sum_{n=1}^\infty \nu_\AAA (E_n)$, where the series is norm convergent and every summand is a compact operator, it follows that $\int_E \AAA \, d\mu = \nu_\AAA (E)$ is a compact operator.
\end{proof}
The condition on $X$ in Theorem \ref{Cpt} is essential, even in the case of finite measure. In fact, we take $X = Y = \ell^1$, $\Omega = \N$, $\mu(\{ n \}) = 2^{-n}$  and $\AAA_n = 2^ne_n \otimes e_n^*$. 
Then we have $\mu(\Omega) = 1$, and the generated o.\,v. measure is not countably additive and the integral over $\Omega$ is not compact although $\AAA_n$ is compact for every $n \in \N$. Moreover these operators commute.

If in addition $X=Y=\HH$ is a Hilbert space we get $\cci$-Pettis integrability.

\begin{corollary}\label{petisCii}
Let $\AAA\in L_s^1(\Omega,\M,\mu,\BH)$ be such that $\AAA_t\in\cci$ for all $t\in\Omega$. Then $\AAA\in\PP(\WWW,\M,\mu,\cci)$. 
\end{corollary}
\begin{proof}
As $\HH$ is a Hilbert space, it does not contain $\ell^1$, thus $\int_E\AAA d\mu\in\cci$ for all $E\in\M$ due to Theorem \ref{Cpt}. Strong integrability of $\AAA$ implies weak$^*$ $\BH$-integrability which is the same as weak $\cci$-integrability. Therefore, $\AAA$ is $\cci$-Pettis integrable.
\end{proof}
The converse of the previous Corollary does not hold, even in the case of finite measure space. Namely, there is a $\cci$-valued Pettis integrable function $\AAA$ that is not strongly integrable, where the operators $\AAA_t$ are positive and mutually commuting.

\begin{example}\label{ContEx} 
Let $(e_n)_{n=1}^\iii$ be an o.\,n.\,b. of $\HH$ and define $\AAA\colon[0,1]\rightarrow\BH$ by 
$$\AAA_t=\sum\limits_{n=1}^{\infty} n^\frac{3}{2}\chi_{\left(\frac{1}{n+1},\frac{1}{n}\right)}(t)\,e_n\otimes e_n^*,\qquad t\in[0,1].$$
Since for each $0 \leqslant t \leqslant 1$ the above sum reduces to one summand each operator $\AAA_t$ is compact, since they are diagonalized by the same o.\,n.\,b. they commute and clearly they are positive.
It is easy to verify that $\AAA\in\GG([0,1],m,\BH)$ and 
$$\int_{[0,1]}\AAA \,dm=\sum\limits_{n=1}^{\infty}\frac{\sqrt{n}}{n+1}\cdot e_n\otimes e_n^*\in\cci.$$
It follows from the proof of \cite[Theorem~2.3, (b)]{MMS} that the family $\AAA$ belongs to $\PP([0,1],m,\cci)$.
Let us prove that the family $\AAA$ is not in $L_s^1([0,1],m,\BH)$. To prove this, note that for all $x\in \HH$
we have the equalities
\begin{equation*}
    \begin{split}
 \int_{[0,1]} \|\AAA_tx\|_\HH \,dm(t) &= \sum_{n=1}^\infty \int_{\left(\frac{1}{n+1},\frac{1}{n}\right)}\|\AAA_tx\|_\HH \, dm(t) =
\sum_{n=1}^\infty \int_{\frac{1}{n+1}}^{\frac{1}{n}} \left\| n^{\frac{3}{2}}( e_n \otimes e_n^*) x\right\|_\HH  dt\\
=   &\sum_{n=1}^\infty \int_{\frac{1}{n+1}}^{\frac{1}{n}} n^{\frac{3}{2}} \cdot |\langle x, e_n \rangle | \,dt = \sum_{n=1}^\infty 
\frac{\sqrt{n}}{n+1}\cdot |\langle x, e_n \rangle |. 
    \end{split}
\end{equation*}
Since $\left(\frac{\sqrt{n}}{n+1}\right)_{n=1}^\iii\notin\ell^2$, the above sum can not be finite for all $x \in \HH$. 
\end{example}

Now, we will generalize the result in Corollary \ref{petisCii} to some classes of Banach spaces that are not necessarily Hilbert. We will focus on the case where the dual of $\KXY$ is representable as a projective tensor product. The following proposition gives sufficient conditions for such representation to be valid. It can be found in \cite[16.7]{DF}. For more details about approximation property and Radon-Nikodym property see, for example \cite{DU}.

\begin{proposition}\label{duality}
Assume $X$ and $Y$ are Banach spaces such that $X^{**}$ or $Y^*$ has approximation property and $X^{**}$ or $Y^*$ has the Radon-Nikodym property. Then, $(\KXY)^*\cong\XzzYz$, where the duality pairing is given by the expression 
$\scal{A}{T}=\sumn\scal{A^*y_n^*}{x_n^{**}}$ for $A\in\KXY$ and $T=\sumn x^{**}_n\otimes y^*_n\in\XzzYz$
where $\sum_{n=1}^\infty \| x_n^{**} \|_{X^{**}} \cdot \| y^*_n \|_{Y^*} < +\infty$.
\end{proposition}

\begin{lemma}\label{fatuzvezda}
Let $\AAA\in L^1_s(\WWW,\M,\mu,\BXY)$, where $X$ is separable and does not contain $\ell^1$. Then, for all $x^{**}\in B_{X^{**}}$ and $y^*\in Y^*$, the function $\scal{\AAA^*y^*}{x^{**}}$ is integrable and and the following equality holds
\begin{equation}\label{XiXstarstar}
\sup_{\| x^{**} \|_{X^{**}} \leqslant 1} \int_\WWW|\scal{\AAA^*_ty^*}{x^{**}}|\,d\mu(t) = \sup_{\|x\|_X\leqslant1}\int_\WWW|\scal{\AAA_tx}{y^*}|\,d\mu(t).
\end{equation}
\end{lemma}
\begin{proof} 
Let $(x^{**}, y^*) \in B_{X^{**}} \times Y^*$ and select an approximating sequence $(x_n)_{n=1}^\iii$ as in Theorem \ref{OdellRosenthal}. Since
$\langle \AAA_t^* y^*, x^{**} \rangle = \lim_{n \to \infty} 
\langle x_n, \AAA_t^* y^* \rangle = \lim_{n \to \infty} \langle \AAA_t x_n, y^* \rangle$ for all $t\in\WWW$, the function $\scal{\AAA^*y^*}{x^{**}}$ is $\mu$-measurable as the pointwise limit of $\mu$-measurable functions $\scal{\AAA x_n}{y^*}$. Moreover, using Fatou lemma we get
\begin{equation*}
\begin{split}&\int_\WWW|\scal{\AAA^*_ty^*}{x^{**}}|\,d\mu(t)=\\ &=\int_\WWW\lim_{n\rightarrow\iii}|\scal{x_n}{\AAA^*_ty^*}|\,d\mu(t) \leqslant\varliminf_{n\rightarrow\iii}\int_\WWW|\scal{x_n}{\AAA^*_ty^*}|\,d\mu(t)\\&\leqslant\sup_{n\in\N}\int_\WWW|\scal{\AAA_tx_n}{y^*}|d\mu(t)\leqslant\sup_{\|x\|_X\leqslant1}\int_\WWW|\scal{\AAA_tx}{y^*}|\,d\mu(t) \leqslant \| y^* \|_{Y^*} \cdot \| \AAA \|_1 < \infty,
\end{split}
\end{equation*}
proving the nontrivial inequality needed to establish \eqref{XiXstarstar}.
\end{proof}

The next lemma is also a technical result which, under appropriate conditions, gives an expression for $\scal{\left(\int_E\AAA^* d\mu\right)y^*}{x^{**}}$. 

\begin{lemma}\label{IntA*y*x**}
    Let $\AAA\in L^1_s(\WWW,\M,\mu,\BXY)$, where $X$ is separable and does not contain $\ell^1$. Then,  $\scal{\AAA^* y^*}{x^{**}}\in L^1(\Omega,\mu)$ for all $(x^{**}, y^*) \in X^{**} \times Y^*$ and 
    \begin{equation*}
\scal{\left(\int_E\AAA^* d\mu\right)y^*}{x^{**}}=\int_E\scal{\AAA^*_ty^*}{x^{**}}d\mu(t), \qquad E \in \M.
    \end{equation*}
\end{lemma}
\begin{proof} 
The existence of $\int_E\AAA^* d\mu$ is provided by Lemma \ref{adjungovani}. The functions
$\AAA^*$ and $\AAA$ have the same support, hence by Lemma \ref{sigfin} we can assume $\mu$ is a $\sigma$-finite measure. Assume $\mu(E)<+\iii$
    and let $(x^{**}, y^*) \in X^{**} \times Y^*$. Then, we have
\begin{equation*}
\begin{split}\scal{\left(\int_E\AAA^*\, d\mu\right)y^*}{x^{**}}&=\lim_{n\rightarrow\iii} \scal{x_n}{\left(\int_E\AAA^*\, d\mu\right)y^*}=\lim_{n\rightarrow\iii}\int_E\scal{x_n}{\AAA^*_ty^*}d\mu(t)\\&=\int_E\lim_{n\rightarrow\iii}\scal{x_n}{\AAA^*_ty^*}\,d\mu(t)=\int_E\scal{\AAA^*_ty^*}{x^{**}}d\mu(t),
\end{split}
\end{equation*}
where the interchange of limit and integral is provided by Vitali's Convergence Theorem. Indeed, the sequence of functions $\psi_n(t)=\scal{x_n}{\AAA^*_ty^*}=\scal{\AAA_tx_n}{y^*}$ is uniformly integrable in $L^1(E,\mu)$, because $|\psi_n(t)|\leqslant\varphi_n(t)=\|\AAA_t x_n\|_Y$ and the family $\{\varphi_n:n\in\N\}$ is uniformly integrable in $L^1(E,\mu)$, due to Corollary \ref{UINT}.
    
If $\mu(E)=+\iii$ we have $E=\bigsqcup_{n=1}^\iii E_n$, where $\mu(E_n)<+\iii$ for all $n\in\N$. Then
\begin{equation*}
\begin{split}
&\scal{\left(\int_E\AAA^* \,d\mu\right)y^*}{x^{**}}=\scal{(\nu_{\AAA^*}(E))y^*}{x^{**}}=\scal{\left(\sumn\nu_{\AAA^*}(E_n)\right)y^*}{x^{**}}\\&=\sumn\scal{(\nu_{\AAA^*}(E_n))y^*}{x^{**}}=\sumn \scal{\left(\int_{E_n}\AAA^* \,d\mu\right)y^*}{x^{**}}\\&=\sumn\int_{E_n}\scal{\AAA^*_ty^*}{x^{**}}\,d\mu(t)
=\int_E\scal{\AAA^*_ty^*}{x^{**}}d\mu(t),
        \end{split}
    \end{equation*}
where the second and third equality follow from Corollary \ref{AzvezdaMera}. The fifth equality holds due to already proven part of the Lemma in the finite measure case. The function $\scal{\AAA^* y^*}{x^{**}}$ belongs to the space $L^1(\Omega,\mu)$, due to the Lemma \ref{fatuzvezda}, which allows us to interchange infinite sum and Lebesgue integral in the last step of the above computation.
\end{proof}

\begin{theorem}\label{DualK}
Assume $X$ and $Y$ are Banach spaces such that $X^{**}$ or $Y^*$ has approximation property and $X^{**}$ or $Y^*$ has the Radon-Nikodym property. In addition, assume $X$ is separable and does not contain a copy of $\ell^1$. If $\AAA$ belongs to $L^1_s(\WWW,\mu,\mathcalb{B}(X,Y))$ and $\AAA_t\in \mathcalb K (X, Y)$ for a.\,e. $t\in\WWW$ then $\AAA \in \PP(\WWW,\mu, \mathcalb K (X, Y))$ and Pettis norm of $\AAA$ is equal to 
\begin{equation}\label{PnormA}
\|\AAA\|_{\PP} = \sup_{\substack{\|x\|_X\leqslant1\\\|y^*\|_{Y^*}\leqslant1}} \int_\Omega 
|\scal{\AAA_t x}{y^*}|\,d\mu(t).
\end{equation}
\end{theorem}
\begin{proof}
 For arbitrary $T$ in $X^{**} \widehat{\otimes}_\pi Y^*$ there is a representation $T = \sumn x_n^{**}\otimes y_n^*$ with $\sumn\|x_n^{**}\|_{X^{**}}\cdot\|y^*_n\|_{Y^*} < +\iii$. Then, using Lemma \ref{fatuzvezda}, we get
 \begin{equation*}
 \begin{split}
\int_\WWW|\scal{\AAA_t}{T}|\,d\mu(t) & = \int_\WWW\left|\sumn\scal{\AAA_t^*y_n^*}{x_n^{**}} \right|d\mu(t) \leqslant \sumn \int_\WWW|\scal{\AAA_t^*y_n^*}{x_n^{**}}|\,d\mu(t) \\ & \leqslant \sumn \|x_n^{**}\|_{X^{**}} 
\sup_{\| x^{**} \|_{X^{**}} \leqslant 1} \int_\WWW |\scal{\AAA_t^*y_n^*}{x^{**}}|\, d\mu(t)\\ & \leqslant \sumn \|x_n^{**}\|_{X^{**}} 
\sup_{\| x \|_X \leqslant 1}
\int_\WWW |\scal{\AAA_t x}{y_n^*}|\, d\mu(t) \\ 
& \leqslant \|\AAA\|_1 \sumn\|x_n^{**}\|_{X^{**}}\|y^*_n\|_{Y^*}<+\iii,
\end{split}
\end{equation*}
and this proves weak integrability of $\AAA$. Since $X$ is separable, support of $\AAA$ has $\sigma$-finite measure, by Lemma \ref{sigfin}. Thus, for every $E\in\M$ the operator $\int_E\AAA \,d\mu$ (in $L_s^1$ sense) is compact, due to  Theorem \ref{Cpt}. In order to prove that this $L^1_s$-integral is also Pettis integral we choose 
$T = \sumn x_n^{**} \otimes y_n^*$ in $\XzzYz$ and, using Proposition \ref{duality} and Lemma \ref{adjungovani}, we obtain
\begin{multline*}
    \scal{\int_E\AAA\,d\mu}{T}=\sumn\scal{\left(\int_E\AAA\, d\mu\right)^*y_n^*}{x_n^{**}}=\sumn\scal{\left(\int_E\AAA^*\, d\mu\right)y_n^*}{x_n^{**}}\\=\sumn\int_E\scal{\AAA^*_ty^*_n}{x^{**}_n}d\mu(t)=\int_E\sumn\scal{\AAA^*_ty^*_n}{x^{**}_n}d\mu(t)=\int_E\scal{\AAA_t}{T}d\mu(t).
\end{multline*}
Hence, $\AAA\in\PP(\WWW,\M,\mu,\KXY)$.

Proof of formula \eqref{PnormA} follows from the following equalities:
\begin{equation}\label{PetisTensor}
\begin{split}
\|\AAA\|_\PP & = \sup\left\{\int_\WWW|\langle\AAA_t,T\rangle|\,d\mu(t)\,:\,\|T\|_{X^{**}\widehat{\otimes}_\pi Y^*}\leqslant1\right\} \\
& = \sup_{\substack{\|x^{**}\|_{X^{**}} \leqslant1 \\ \|y^*\|_{Y^*} \leqslant1}} \int_\WWW 
|\scal{\AAA^*_ty^*}{x^{**}}|\,d\mu(t) = \sup_{\substack{\|x\|_{X} \leqslant1 \\ \|y^*\|_{Y^*}\leqslant1}} \int_\WWW|\scal{\AAA_tx}{y^*}| \,d\mu(t).
\end{split}
\end{equation}
The first equality in \eqref{PetisTensor} is definition of the Pettis norm $\|\AAA\|_{\PP}$ of the function $\AAA:\Omega\to\KXY$, see \eqref{Pnorm}, proof of the second one is analogous to the proof of \cite[Proposition~1.12]{MS}. The last equality in \eqref{PetisTensor} follows from Lemma \ref{fatuzvezda}. 
\end{proof}

\section{Subadditivity of the spectral radius for strongly integrable commuting families}

It is well known that for a commuting finite family $(A_j)_{j=1}^n$ in $\mathcalb{B}(X)$ we have:
$$r \left( \sum_{j=1}^n A_j \right) \leqslant \sum_{j=1}^n r(A_j).$$
This is spectral radius inequality.
Commutativity assumption is essential, the above is not true in general even for $n=2$ and $\dim X = 2$. The above inequality remains true if we have a commuting infinite family of operators $(A_n)_{n\in\N}$ such that the series $\sum_{n=1}^{\infty}A_n$ converges in the norm of $\mathcalb{B}(X)$: 
\begin{lemma}\cite[Lemma 2.7]{HMspectralr}\label{sprad}
    If $\sum_{n=1}^\infty A_n$ is a norm convergent sum of mutually commuting operators in $\mathcalb{B} (X)$, then 
    $$r \left( \sum_{n=1}^\infty A_n \right) \leqslant \sum_{n=1}^\infty r(A_n).$$
\end{lemma}
In the case of counting measure $\mu$ on $\Omega=\N$ strong integral $\int_\N\AAA_n\,d\mu(n)$ is the strong (absolutely convergent) sum $\prescript{\raisebox{-1.5ex}{$\scriptstyle s$}}{}{\sum_{n=1}^\infty}\AAA_n$ of the sequence $(\AAA_n)_{n=1}^\infty$ of operators in $\mathcalb B(X)$. 
The above lemma and Theorem \ref{l1sbuk} give us the following corollary.
\begin{corollary}\label{JakaBrojacka}
If $X^*$ does not contain $\mathfrak{c}_0$ and $\AAA=(\AAA_n)_{n\in\N}$ is a commuting family in $\ell_s^1(\mathcalb B(X))$ then
\begin{equation}\label{rsum}
r\left({\sum_{n=1}^\infty}\AAA_n\right)\leqslant\sum_{n=1}^\infty r(\AAA_n).
\end{equation}
\end{corollary}

In the case of strong sums, without condition $\AAA\in\ell_s^1(\mathcalb B(X))$ we can not obtain inequality \eqref{rsum}, as the next example shows.

\begin{example}
    Let $\HH = L^2(0,1)$ and let $V$ be the classical Volterra operator, i.\,e. $$Vf(x) = \int_{(0,x)} f(t)\,dm(t), \qquad \,0<x<1.$$ For all non negative integers $n$ we set $$\BBB_n = I - (I + nV)^{-1} = nV(I+nV)^{-1},\qquad\AAA_n = \BBB_n - \BBB_{n-1}.$$Note that the family $\AAA = (\AAA_n)_{n=1}^\infty$ consists of Volterra operators which commute with each other. We claim that 
\begin{equation}\label{Voltsum}
\prescript{\raisebox{-1.7ex}{$\scriptstyle s$}}{}{\sum_{n=1}^\infty}\AAA_n = I.
\end{equation}
Clearly, $\BBB_n$ is the $n$-th partial sums of the above series and therefore \eqref{Voltsum} is equivalent to strong convergence of $$\CCC_n = I - \BBB_n = (I + nV)^{-1}$$to the zero operator. Since the equality $\CCC_n f = g$ is equivalent to 
$$ f(t) = g(t) + n\int_{(0,t)} g(s) \, dm(s), \qquad 0 < t < 1, $$
one obtains, solving a Volterra integral equation of the second type, the following formula for $\CCC_n$:
$$\CCC_n f(t) = f(t) - n \int_{(0,t)} e^{-n(t-s)} f(s) \, dm(s), \qquad f \in L^2(0,1).$$
On the right hand side there is a convolution of $f$ with a positive function $\varphi_n$, where $\varphi_n(t) = n\varphi_1(nt)$ and 
$$\varphi_1(t) = \begin{cases}
0, & \text{if}\ t< 0, \\
e^{-t}, & \text{if}\,\,t \geqslant 0.
\end{cases}$$
Since $\|\varphi_1 \|_{L^1} = 1$, $(\varphi_n)_{n=1}^\infty$ is an approximation of identity. Hence, we have $L^2$ convergence of the corresponding integrals to $f$. Thus $\lim_{n\to\infty}\|\CCC_nf\|_{L^2(0,1)}=0$ for all $f\in L^2(0,1)$ i.\,e. the sequence of operators $(\CCC_n)_{n=1}^\infty$ strongly converges to zero and hence we have proven \eqref{Voltsum}. Now, we have
$$1=r(I)=r\left(\prescript{\raisebox{-1.5ex}{$\scriptstyle s$}}{}{\sum\limits_{n=1}^\infty}\AAA_n \right) >\sum_{n=1}^\infty r(\AAA_n)=0,$$
since $r(\AAA_n)=0$ for all $n\in\N$. Indeed, for all $n\in\N$ we have $$r(\AAA_n)\leqslant r((I + nV)^{-1})\cdot r(V)\cdot r((I + (n-1)V)^{-1})=0,$$ due to the mentioned commutativity and $r(V)=0$.
\hfill$\triangle$
\end{example}

It is natural to ask if there is analogous statement for integrals instead of sums of operators. An affirmative answer was given in \cite{HMspectralr}, where it was proved that
\begin{equation*}
r\left(\int_\Omega \AAA\,d\mu\right)\leqslant \int_\Omega r(\AAA_t)\,d\mu(t), 
\qquad \AAA \in L^1(\Omega, \mu, \mathcalb{B}(X)).
\end{equation*}
We note that the argument presented in \cite{HMspectralr} shows that the above inequality holds for Bochner integral of a commuting family of elements of a Banach algebra.

In this section we extend this inequality from the class of Bochner integrable functions to a subclass of strongly integrable functions, see Theorem \ref{thradijus} below.

We start with a simple lemma on commutativity of strong integrals.

\begin{lemma}\label{IntAB}
Let $\AAA,\BBB\in L_s^1(\Omega,\M,\mu,\mathcalb B(X))$ and assume $\AAA_t\BBB_s=\BBB_s\AAA_t$
for $\mu$-almost all $t,s\in\Omega$. Then, for all $E,F\in\M$ we have 
$$\int_E\AAA\,d\mu\int_F\BBB\,d\mu=\int_F\BBB\,d\mu\int_E\AAA\,d\mu.$$ 
\end{lemma}
\begin{proof}
Since $\AAA$ and $\BBB$ belong to $L_s^1(\Omega,\mu,\mathcalb B(X))$, we can define operators  
$$A:=\int_E\AAA\,d\mu\in\mathcalb B(X)\,\,\,\text{and}\,\,\,B:=\int_F\BBB\,d\mu\in\mathcalb B(X).$$For $\mu$-almost all $t\in\Omega$ we have the equality $A\BBB_t=\BBB_t A$. Indeed, for fixed $t\in\Omega$, using \cite[Proposition 2.25]{MAMK} we have

\begin{equation*}
    \begin{split}
A\BBB_t&=\left(\int_E\AAA\,d\mu\right)\BBB_t=\int_E\AAA_s\BBB_t\,d\mu(s)\\&=\int_E\BBB_t\AAA_s\,d\mu(s)=\BBB_t\left(\int_E\AAA\,d\mu\right)=\BBB_tA.
    \end{split}
\end{equation*}
Using this equality and again \cite[Proposition 2.25]{MAMK} we obtain: 

\begin{equation*}
\begin{split}
AB=A\left(\int_F\BBB\,d\mu\right)=\int_F A\BBB_t\,d\mu(t)=\int_F\BBB_tA\,d\mu(t)=\left(\int_F\BBB\,d\mu\right) A=BA,
\end{split}
 \end{equation*}
and this completes the proof.
\end{proof}
The above lemma can also be proved for families in $L_w^1(\Omega,\mu,\mathcalb B(X))$ using the same arguments, provided $X$ is reflexive. Indeed, under this assumption, the operators $A$ and $B$ appearing in the above proof belong to $\mathcalb B(X)$.

We need some measurability results and we start with the following observation: if $\AAA \in \mathcal M_w(\Omega, \M, \mu, \mathcalb B(Y, Z))$ and $(y_t)_{t\in\Omega}$ is a strongly $\mu$-measurable family in $Y$, then $(\AAA_t y_t)_{t \in \Omega}$ is a weakly $\mu$-measurable family in $Z$. 
To see that, choose a sequence of $\mu$-simple functions $s_n : \Omega \to Y$ such that $\lim_{n\to\infty} s_n(t) = y_t$ for $\mu$-almost all $t\in\Omega$. Then, for every $z^* \in Z^*$, we have $\langle \AAA_t y_t, z^* \rangle = \lim_{n \to \infty} \langle \AAA_t s_n(t), z^* \rangle$. Since the scalar function $\Omega\ni t\mapsto\langle \AAA_t s_n(t), z^* \rangle\in\C$ is $\mu$-measurable for every $n \in \N$, our claim is proven. 

\begin{lemma}
    Let $\AAA^{(1)} \in \mathcal M_w(\Omega, \M, \mu, \mathcalb B(X))$ and $\AAA^{(j)} \in \mathcal M_s(\Omega, \M, \mu, \mathcalb B(X))$ for $j\in\{2,\ldots,n\}$. 
    Then $\AAA^{(1)} \AAA^{(2)} \cdots \AAA^{(n)} \in \mathcal M_w(\Omega, \M, \mu, \mathcalb B(X))$.
\end{lemma}
\begin{proof}
It suffices to prove the statement for $n=2$, then the proof is completed by an inductive argument. So we assume $n=2$ and fix $x \in X$. Then weak $\mu$-measurability of $\AAA^{(1)} \AAA^{(2)} x$ follows from the above observation applied to the family $y_t = \AAA^{(2)}_t x$.
\end{proof}

In the preceding part of this section, separability of the Banach space X was not required. However, for the following statement, which partially solves Problem 2.11 from \cite{HMspectralr}, this assumption will be needed

\begin{lemma}\label{rismeas}
Assume $X$ is separable and $\AAA \in \mathcal M_w(\Omega, \M, \mu, \mathcalb B(X))$. Then the function $r_\AAA : \Omega \to \R$ defined by $r_\AAA (t) = r(\AAA_t)$ is $\mu$-measurable.
\end{lemma}
\begin{proof}
    Since $X$ is separable we have the equality (see \eqref{Mws}) $$\mathcal M_w(\Omega, \M, \mu, \mathcalb B(X)) = \mathcal M_s(\Omega, \M, \mu, \mathcalb B(X))$$and thus the above lemma gives $$\AAA^n = (\AAA_t^n)_{t\in\Omega} \in \mathcal M_s(\Omega, \M, \mu, \mathcalb B(X)).$$Hence, using again separability of $X$, 
    $\Omega\ni t \mapsto \| \AAA_t^n \|^{1/n}\in\R$ is measurable for every $n \in \N$ (see comments after the equality \eqref{Mws}), and this suffices since the equality $r(\AAA_t)= \inf\limits_{n\in\N}\| \AAA_t^n \|^{1/n}$ holds for all $t\in\Omega$.
\end{proof}
Strictly speaking, Lemma \ref{rismeas} is not necessary for the proof of Theorem \ref{thradijus}, because $\mu$-measurability of $\Omega\ni t\mapsto r(\AAA_t)\in\R$ in Theorem \ref{thradijus} can be deduced from  assumed strong $\mu$-measurability of the family $\AAA$.

Now, we present a generalization of the main result from \cite{HMspectralr}
 for a strongly $\mu$-measurable family $\AAA\in L_{s}^1(\Omega,\mu,\mathcalb B(X))$ which consists of mutually commuting operators $\AAA_t$.

\begin{theorem}\label{thradijus}
Assume $X^*$ does not contain an isomorphic copy of $\mathfrak{c}_0$ and let $\AAA\in L_{s}^1(\Omega,\mu,\mathcalb B(X))$ be a strongly $\mu$-measurable family of mutually commuting operators. Then, we have the inequality
\begin{equation*}\label{rpb}
r\left(\int_\Omega \AAA\,d\mu\right)\leqslant \int_\Omega r(\AAA_t)\,d\mu(t).\end{equation*}
\end{theorem}
\begin{proof}
We assume at first that $\mu(\Omega)<+\infty$. For all $n\in\N\cup\{0\}$ define the set 
$$\Omega_n=\{t\in\Omega:\|\AAA_t\|\in[n,n+1)\}.$$
Assumed strong $\mu$-measurability of $\AAA$ implies $\mu$-measurability of the real valued function $\Omega\ni t\mapsto\|\AAA_t\|\in\R$. Hence the sets $\Omega_n$ are $\mu$-measurable. Clearly, we have the equality $\Omega=\bigsqcup_{n=0}^{\infty}\Omega_n$. Since $$\int_{\WWW_n}\|\AAA_t\|\,d\mu(t)\leqslant(n+1)\cdot\mu(\WWW_n)<+\iii,$$we conclude $\AAA\in L^1(\WWW_n,\mu,\mathcalb B(X))$ for all $n\in\N\cup\{0\}$ (there we again use assumed strong $\mu$-measura\-bility of $\AAA$). 
Now, we have the next calculation:
\begin{equation}\label{racunr}
    \begin{split}
r\left(\int_\Omega\AAA\,d\mu\right)&=r(\nu_\AAA(\WWW))=r\left(\sum_{n=0}^\iii\nu_\AAA(\WWW_n)\right)\leqslant\sum_{n=0}^{\infty}r\left(\int_{\Omega_n}\AAA\,d\mu\right)\\&\leqslant
        \sum_{n=0}^{\infty}\int_{\Omega_n}r(\AAA_t)\,d\mu(t)=\int_{\Omega}r(\AAA_t)\,d\mu(t).
    \end{split}
\end{equation}
The second equality in \eqref{racunr} is based on Theorem \ref{ACmuPb}, which tells us that $\nu_\AAA$ is a countably additive $\mathcalb B(X)$-valued measure. 
The first inequality in \eqref{racunr} follows from Lemma \ref{IntAB} and Lemma \ref{sprad}. 
For the second inequality in \eqref{racunr} we apply \cite[Theorem 2.8]{HMspectralr}, and here we use assumed strong $\mu$-measurability of $\AAA$ to conclude Bochner integrability of $\AAA$ on each set $\Omega_n$.

Now we consider the general case. Since $\AAA$ has $\sigma$-finite support
we can assume $\WWW = \bigsqcup_{n=1}^\iii \Omega_n$, where $\mu(\Omega_n) < +\iii$ for all $n\in\N$. Then, by already proved part of theorem for finite measures we have the inequality
$$r\left(\int_{\Omega_n} \AAA\, d\mu\right) \leqslant \int_{\Omega_n} r(\AAA_t)\,d\mu(t), \qquad n \in \N.$$
Now, similarly to the calculations in \eqref{racunr}, we get 
$$r\left(\int_\WWW\AAA \,d\mu\right)\leqslant\sum_{n=1}^\iii \int_{\Omega_n}r(\AAA_t)\,d\mu(t)=\int_\WWW r(\AAA_t)\,d\mu(t).\qedhere$$
\end{proof}

In the next statement we again use the F.\,D.\,D. property to derive strong $\mu$-measurability of o.\,v. function $\AAA$. 


\begin{corollary}
Assume $X$ has a shrinking F.\,D.\,D., $\AAA \in  L_s^1(\Omega,\mu,\mathcalb B(X))$ is a commuting family of compact operators. Then, we have the following inequality
$$r\left(\int_\WWW\AAA\, d\mu\right)\leqslant\int_\WWW r(\AAA_t)\,d\mu(t).$$
Moreover, if $\AAA_t$ are Volterra operators, so is $\int_\WWW\AAA d\mu$.
\end{corollary}
\begin{proof}
    By Proposition \ref{MerljivAprox} we get that $\AAA:\Omega\to\mathcalb B(X)$ is strongly $\mu$-measurable function. Furthermore, as $X^*$ is separable (because $X$ has a shrinking F.\,D.\,D.) it cannot contain an isomorphic copy of $\ell^\iii$ and thus cannot contain an isomorphic copy of $\mathfrak{c}_0$, due to \cite[Corollary~6, Page 23]{DU}. 
Final conclusion now follows from Theorem \ref{thradijus}. If $\AAA_t$ is a Volterra operator for all $t \in \Omega$, then due to proved inequality we have that the operator $\int_\Omega\AAA\, d\mu$ is Volterra too.
\end{proof}

Concluding remarks: It is not known to the authors whether the strong $\mu$-measurability of the family $\AAA$ in Theorem \ref{thradijus} can be replaced by a weaker assumption. We also note that Lemma \ref{rismeas} provides alternative conditions for $\mu$-measurability of function $\Omega\ni t\mapsto r(\AAt)\in\R$.



\section*{Acknowledgements}
The authors of this research are partially supported by the Ministry of  Science, Technological Development and Innovation, Republic of Serbia, through the project \texttt{451-03-33/2026-03/200104}.

\end{document}